\documentclass[11pt, a4paper]{amsart}
\usepackage{amsfonts,amssymb,amsmath, amsthm}
\usepackage{mathrsfs}
\usepackage{lmodern}
\usepackage{qsymbols}
\usepackage{latexsym}
\usepackage[noadjust]{cite}
\usepackage{pdfsync}
\usepackage{bm}
\usepackage[cmtip,all]{xy}
\usepackage[french,english]{babel}
\usepackage{enumitem}

\newtheorem{thm}{Theorem}[section]
\newtheorem{lem}[thm]{Lemma}
\newtheorem{prop}[thm]{Proposition}
\newtheorem{cor}[thm]{Corollary}
\theoremstyle{definition}

\newtheorem{rem}[thm]{Remark}
\numberwithin{equation}{section}
\usepackage[plainpages=false,pdfpagelabels, citecolor=red]{hyperref}
\usepackage{xcolor}
\hypersetup{
 colorlinks,
 linkcolor={cyan!90!black},
 citecolor={magenta},
 urlcolor={green!40!black}
} 
\newcommand{\R}{\mathbb{R}}
\newcommand{\IC}{\mathbb{C}}

\newcommand{\IZ}{\mathbb{Z}}
\newcommand{\cS}{\mathcal{S}} 
\newcommand{\cF}{\mathcal{F}} 
\newcommand{\cL}{\mathcal{L}}
\newcommand{\loc}{\operatorname{loc}}
\renewcommand{\L}{\operatorname{L}} 
\newcommand{\Lloc}{\L_{\operatorname{loc}}} 
\newcommand{\C}{\operatorname{C}} 
\renewcommand{\H}{\operatorname{H}} 
\newcommand{\W}{\operatorname{W}}

\newcommand{\ree}{{\mathbb{R}^{n+1}}}

\renewcommand{\div}{\operatorname{div}}
\newcommand{\dhalf}{D_t^{1/2}} 
\newcommand{\HT}{H_t} 
\let\ii\i
\renewcommand{\i}{\mathrm{i}} 
\newcommand{\eps}{\varepsilon} 
\renewcommand{\d}{\mathrm{d}}
\renewcommand\Re{\operatorname{Re}}

\newcommand{\blank}{\mathrel{\,\cdot\,}} 
\newcommand{\Lop}{\mathcal{L}} 
\newcommand{\cl}[1]{\overline{#1}} 
\newcommand{\dual}[2]{\langle #1,#2 \rangle}

\newcommand{\sgn}{\operatorname{sgn}}

\newcommand{\con}[1]{\overline{#1}}
\newcommand{\pd}{\partial}
\def\Xint#1{\mathchoice
{\XXint\displaystyle\textstyle{#1}}%
{\XXint\textstyle\scriptstyle{#1}}%
{\XXint\scriptstyle\scriptscriptstyle{#1}}%
{\XXint\scriptscriptstyle%
\scriptscriptstyle{#1}}%
\!\int}
\def\XXint#1#2#3{{\setbox0=\hbox{$#1{#2#3}{%
\int}$ }
\vcenter{\hbox{$#2#3$ }}\kern-.6\wd0}}
\def\barint{\,\Xint -} 
\def\bariint{\barint_{} \kern-.4em \barint}
\def\bariiint{\bariint_{} \kern-.4em \barint}
\renewcommand{\iint}{\int_{}\kern-.34em \int} 
\renewcommand{\iiint}{\iint_{}\kern-.34em \int} 
\makeatletter
\newenvironment{abstracts}{%
  \ifx\maketitle\relax
    \ClassWarning{\@classname}{Abstract should precede
      \protect\maketitle\space in AMS document classes; reported}%
  \fi
  \global\setbox\abstractbox=\vtop \bgroup
    \normalfont\Small
    \list{}{\labelwidth\z@
      \leftmargin3pc \rightmargin\leftmargin
      \listparindent\normalparindent \itemindent\z@
      \parsep\z@ \@plus\p@
      
      \itemsep\medskipamount
    }%
}{%
  \endlist\egroup
  \ifx\@setabstract\relax \@setabstracta \fi
}

\newcommand{\abstractin}[1]{%
  \otherlanguage{#1}%
  \item[\hskip\labelsep\scshape\abstractname.]%
}
\makeatother
\title[]{On regularity of weak solutions to linear parabolic systems with measurable coefficients}
\author{Pascal Auscher}
\author{Simon Bortz}
\author{Moritz Egert}
\author{Olli Saari}
\address{Laboratoire de Math\'{e}matiques d'Orsay, Univ. Paris-Sud, CNRS, Universit\'{e} Paris-Saclay, 91405 Orsay, France \vspace{5pt}\newline\vspace{5pt}{\rm and}\newline Laboratoire Ami\'{e}nois de Math\'{e}matique Fondamentale et Appliqu\'{e}e, UMR 7352 du CNRS, Universit\'{e} de Picardie-Jules Verne, 80039 Amiens, France}
\email{pascal.auscher@math.u-psud.fr}

\address{Laboratoire de Math\'{e}matiques d'Orsay, Univ. Paris-Sud, CNRS, Universit\'{e} Paris-Saclay, 91405 Orsay, France}
\email{moritz.egert@math.u-psud.fr}

\address{School of Mathematics, University of Minnesota, Minneapolis, MN 55455, USA}
\email{bortz010@umn.edu} 

\address{Department of Mathematics and Systems Analysis, Aalto University, FI-00076 Aalto, Finland \vspace{5pt}\newline\vspace{5pt}{\rm and}\newline Mathematical Institute, University of Bonn, 53115 Bonn, Germany
}
\email{saari@math.uni-bonn.de} 

\thanks{The first and third authors were partially supported by the ANR project ``Harmonic Analysis at its Boundaries'', ANR-12-BS01-0013. This material is based upon work supported by National Science Foundation under Grant No.\ DMS-1440140 while the  authors were in residence at the MSRI in Berkeley, California, during the Spring 2017 semester. The second author was supported by the NSF INSPIRE Award DMS-1344235. The third author was supported by a public grant as part of the FMJH}
\subjclass[2010]{Primary: 35K40, 26B35. Secondary: 35A15, 26A33.} 
\date{\today}
\dedicatory{}
\keywords{Parabolic systems, weak solutions, local H\"older regularity, reverse H\"older estimates, self-improvement properties, fractional derivatives}
\begin{document}
\begin{abstracts}
\abstractin{english}
We establish a  new regularity property for weak solutions of linear parabolic systems with coefficients depending measurably on time as well as on all spatial variables. Namely, weak solutions are {locally} H\"older continuous $\L^p$ valued functions for some $p>2$.

\abstractin{french}
On d\'emontre une nouvelle propri\'et\'e de r\'egularit\'e des solutions faibles des syst\`emes paraboliques dont les coefficients d\'ependent de fa\c con mesurable du temps et des variables spatiales. Pr\'ecis\'ement, on montre que ces solutions sont localement H\"older continues comme fonctions \`a valeurs dans un espace $\L^{p}$ pour un $p>2$. 
\end{abstracts}

\maketitle

\section{Introduction}

This work is concerned with local regularity of weak solutions to linear parabolic equations or systems in divergence form, 
\begin{align}
 \label{eq1}
 \partial_t u -\div_{x} A(t,x)\nabla_{x} u = f+\div_{x} F,
\end{align}
in absence of any regularity of the coefficients besides measurability. The system is considered in an open parabolic cylinder $(t,x) \in \Omega \subseteq \R \times \R^n$, $n \geq 1$, ellipticity is imposed in the sense of a weak G\aa{}rding inequality and weak solutions belong to the usual Lions class, that is to say, $u$ and $\nabla_{x} u$ are locally square-integrable. We note at this stage that we do not impose solutions to be locally bounded in any sense.
 
The topic has a long history, probably starting with the famous results of Nash~\cite{Nash} and Moser~\cite{Moser} that weak solutions to parabolic equations with real coefficients are H\"older continuous with respect to the parabolic distance. This is not true for equations with complex coefficients let alone systems, even in dimension $n = 2$. First counterexamples were constructed by Frehse--Meinel~\cite{FM} in dimension $n \geq 3$ and very recently by Mooney~\cite{Mooney} in dimension $2$.
Results obtaining H\"older estimates beyond the Nash--Moser theorem have mostly focused on either quasilinear equations and systems with more regular coefficients, notably $\C^1$-regular in time and space~\cite{BF1} or absolutely continuous in time~\cite{FS-N}, and systems with a very specially structured real coefficient matrix such as the diagonal systems in~\cite{BF2}. As for most general linear parabolic systems, which are considered in this work, Lions~\cite{Li57} showed continuity in time valued in spatial $\Lloc^2$. This was improved to $\Lloc^p$ for some $p>2$ in the case of real coefficients by Ne\v cas--Sver\'ak~\cite{Necas-Sverak}, and Giaquinta--Struwe obtained the local higher integrability of $\nabla_x u$ in \cite{GS}. In this paper, we reveal a new regularity property of weak solutions. In its simplest form it can be stated as follows.

\begin{thm} 
\label{thm:holder} 
If $u$ is a local weak solution to the homogeneous system $\partial_t u -\div_{x} A(t,x)\nabla_{x} u = 0$, then in time $u$ is locally bounded and H\"older continuous with values in spatial $\Lloc^p$ for some $p>2$. 
\end{thm}

Most of the regularity properties of solutions to parabolic systems have been established through the local variational methods emerging from the Lions theory \cite{Li57}. It does not seem that those methods give access to our result. Instead we rely on a global variational approach based  on this simple observation: 
We can extend the local solution $u$ via multiplication with a smooth cut-off to a global function $v$ and study the corresponding inhomogeneous problem
\begin{align}
 \label{eq2}
 \partial_t v -\div_{x} A(t,x)\nabla_{x} v = \widetilde f+\div_{x} \widetilde F,
\end{align}
now on all of $\ree$. Indeed, any local information of $v$ carries over to $u$, but the global setup with the real line describing the time enables to bring powerful tools such as singular integral operators and the Fourier transform into play. Most notably, splitting $\partial_t = \dhalf \HT \dhalf$ according to the Fourier decomposition $\i \tau = |\tau|^{1/2} (\i \sgn(\tau)) |\tau|^{1/2}$, there is a sesquilinear form
\begin{align}
\label{eq:sesqui}
 a(v,\phi) = \iint_\ree A \nabla_{x} v \cdot \cl{\nabla_{x} \phi} + \HT \dhalf v \cdot \cl{\dhalf \phi} \; \d x \, \d t
\end{align}
corresponding to \eqref{eq2} which, in contrast to the Lions theory on finite time intervals, admits a hidden coercivity on a natural energy space. This uses the algebraic properties of the Hilbert transform on the real line in a crucial way. See Lemma~\ref{lem:2} below for details. These observations are not new but it does not seem they have been fully exploited for obtaining local regularity of solutions up to now. To the best of our knowledge, the idea first appeared in the work of Kaplan~\cite{Kaplan}. This idea was rediscovered by Hofmann--Lewis~\cite{HL} in the context of parabolic boundary value problems (see also \cite{AEN} and the references therein) and has recently played a role in the realm of non-autonomous maximal regularity~\cite{AE, Dier-Zacher}. Half-order derivatives on finite time intervals were also studied in~\cite[Ch.~VII]{LSU}.

Compared to the local variational approach, where the $t$-derivative is understood in a weak sense only through the equation (see Section~\ref{sec:notation}),  we can 
now study the exact amount of differentiability that $v$ should admit on the global level through a locally integrable function -- the fractional derivative $\dhalf v(t,x)$. In fact, we have \emph{a priori} a parabolic differential 
\begin{align*}
 |Dv| := |\nabla_{x} v|+  |\HT \dhalf v| + |\dhalf v|+ |v| \in \L^2(\R ; \L^2(\R^n)).
\end{align*}
Since time is a one-dimensional variable, square-integrability of $\dhalf v$ is already the borderline case from the view-point of Sobolev embeddings,  not enough for continuity in time though, which probably explains why $\dhalf v$ has not been exploited before. On the other hand, higher integrability of $\dhalf v$ would yield H\"older continuity in time. Indeed most of this paper is devoted to establishing the self-improvement of integrability for the spatial gradient and the half-order time derivative simultaneously, that is to say,
\begin{align}
\label{eq:improvement}
 |Dv| \in \L^p(\R ; \L^p(\R^n)), \quad \text{for some $p>2$}.
\end{align}
We  present two proofs relying on rather different methods, both using the global variational formulation explained above. We think they each have their own interest, with potential applicability to non-linear systems for the first one and to other types of parabolic equations as well as fractional elliptic equations for the second one. 

\subsection{Strategy of the proofs}

In Section~\ref{sec:proof1} we present a real analysis proof of \eqref{eq:improvement}. The idea is to use, as in the analogous result for elliptic equations \cite{Meyers}, self-improvement properties of reverse H\"older inequalities known as Gehring's lemma. First, we prove a new and delicate reverse H\"older inequality for $Dv$, by extending ideas from \cite{AEN}. The non-locality of the fractional derivative reflects in local $\L^2$ averages of $Dv$ being controlled only by a weighted \emph{infinite} sum of $\L^1$ averages. 
Hence, we need a substantial extension of the classical Gehring lemma, which we shall prove in Section~\ref{sec:Gehring} and could be of independent interest. An unrelated Gehring type lemma ``with tail'' recently obtained in the context of fractional elliptic equations \cite{KMS} has been inspiring to us. We mention that we shall explore further such extensions in a forthcoming work \cite{ABES2017}. For other modifications of the local parabolic Gehring lemma we refer to \cite{Gehring-Parabolic} and references therein.

In Section~\ref{sec:proof2} we present an operator theoretic proof of \eqref{eq:improvement}. We consider the operator $\cL$ associated with the sesquilinear form \eqref{eq:sesqui} in virtue of the Lax-Milgram lemma and use an analytic perturbation argument. More precisely, exploiting the hidden coercivity in a crucial way, $\cL$ plus a large constant turns out to be invertible from the natural $\L^2$ energy space to its dual and extends boundedly to the corresponding $\L^p$-based spaces. The higher integrability of $Dv$ then follows from the fact that invertibility of a bounded operator between complex interpolation scales  extrapolates. The latter is known as {\v{S}}ne{\u\ii}berg's theorem~\cite{Sneiberg-Original}.

\subsection{Main results}

All this is for homogeneous equations so far, that is $f=0$ and $F=0$. However, as the extension to $\ree$ forces us to work with inhomogeneous equations anyway, there is no real obstacle to start with an inhomogeneous equation right away. Here we give an informal description of our main results. Precise theorems and their proofs are found in Sections~\ref{sec:localEstimates} and \ref{sec:proof1} - \ref{sec:proof3}.

We consider weak solutions in $\Omega = I_0 \times Q_0$ to \eqref{eq1} under the assumptions $f\in \Lloc^{q'}(\Omega)$ and $F\in \Lloc^2(\Omega)$, where $q= 2 + 4/n$ and $q'$ is its H\"older conjugate. 

It is worth mentioning that $q'<2$ and hence we go beyond the usual Lions variational approach that uses $\partial_{t}u\in \Lloc^2(I_{0}; \W_{\loc}^{-1,2}(Q_{0}))$ as a starting point to obtain in-time continuity of $u$ valued in spatial $\Lloc^2$. We could have made $u\in \Lloc^\infty(I_{0}; \Lloc^2(Q_{0}))$ an assumption as for instance in \cite{NW} but in fact -- and this is an observation we have not found in the literature -- this uniform local $\L^2$ boundedness in space and the local $\L^q$ integrability both follow from the hypotheses, again thanks to the global variational approach and the use of half-order derivatives. More precisely, we are still able to obtain the ``classical'' cornerstones of the Lions theory:

\begin{enumerate}
 \item A Caccioppoli inequality (Proposition~\ref{ref:cac}).
 \item In-time continuity of $u$ with values in spatial $\Lloc^2$ (Theorem~\ref{thm:reg-i}).
 \item Higher $\L^q$-integrability for $u$ in time and space through a reverse H\"older inequality (Proposition~\ref{ref:basic}).
\end{enumerate}

Next, if $p>2$ is sufficiently close to $2$, depending only on ellipticity and dimensions, then we have under the assumptions $f \in \Lloc^{p_*}(\Omega)$ and $F \in \Lloc^p(\Omega)$ the following improvements. Here, $p_* = (n+2)p/(n+2+p)$.

\begin{enumerate}
\setcounter{enumi}{3}
\item $\L^p$-control of $|\dhalf(u \chi)| + |\HT \dhalf(u\chi)|$ for any smooth and compactly supported $\chi$ (Theorem~\ref{thm:reg}).
 \item Control of the $(1/2 - 1/p)$-H\"older modulus of continuity in time of the spatial $\L^p$-norm of $u$ (Theorem~\ref{thm:RH grad}).
  \item Higher $\L^p$-integrability for $|\nabla u|$ through a reverse H\"older inequality (Theorem~\ref{thm:GS}).

 \end{enumerate}
As explained before, (iv) and (v) are completely new in this generality and the main contribution of this work. Moreover, (vi) was first obtained in \cite{GS} when $f=0,F=0$ by means of the classical Gehring lemma and was generalized to non zero right-hand side in \cite{C}, but with stronger requirements on $f$ and $F=0$. Such results have impact on partial regularity of nonlinear systems \cite{GG, GS,C}.

In the following Section~\ref{sec:notation} we introduce relevant notation. The global variational setup using \eqref{eq:sesqui} is discussed afterwards in Section~\ref{sec:i}.

\subsection*{Acknowledgement}
We thank the anonymous referees for pointing out relevant literature.
\section{Notation and basic definitions}
\label{sec:notation}

Most of our notation is standard. One exception is that for $X$ a Banach space we let $X^*$ be the (anti-)dual space of conjugate linear functionals on $X$. For exponents $p \in (1,\infty)$ we define the upper and lower \emph{Sobolev conjugate} $p^*$ and $p_*$ with respect to parabolic scaling and the \emph{H\"older conjugate} $p'$ through the relations
\begin{align*}
 \frac{1}{p^*} = \frac{1}{p} - \frac{1}{n+2}, \qquad \frac{1}{p_*} = \frac{1}{p} + \frac{1}{n+2}, \qquad \frac{1}{p'} = 1 - \frac{1}{p},
\end{align*}
whenever they belong again to the interval $(1,\infty)$. Hence, for the exponent $q$ used above we have $q=2^*$ and $q'=2_*$.
With regard to parabolic systems and their weak solutions, we use the following notions.

\subsection{Ellipticity}

In what follows we assume bounded, measurable, complex valued coefficients
\begin{equation}   
\label{eq:boundedmatrix}
  A(t,x)=(A_{i,j}^{\alpha,\beta}(t,x))_{i,j=1,\ldots, n}^{\alpha,\beta= 1,\ldots,m}\in \L^\infty(\ree;\cL(\IC^{mn}))
\end{equation}  
and that there exist $\lambda >0$ and $\kappa \geq 0$ such that the (weak) G\aa rding inequality 
\begin{equation}   
\label{eq:accrassumption}
  \Re \int_{\R^n} A(t,x)\nabla_x u(x)\cdot  \con{\nabla_x u(x)} \;  \d x\ge \lambda 
   \int_{\R^n} |\nabla_{x} u(x)|^2 \; \d x - \kappa  \int_{\R^n} | u(x)|^2 \; \d x
\end{equation}
holds for all $u \in \W^{1,2}(\R^n; \IC^m)$, uniformly in $t\in \R$. Our notation is 
\begin{align*}
 A(t,x)\nabla_x u(x)\cdot  \con{\nabla_x u(x)} := \sum_{\substack{i,j=1,\ldots, n \\ \alpha,\beta= 1,\ldots,m}} A_{i,j}^{\alpha,\beta}(t,x) \pd_{j}u^\beta(x)\con{\pd_{i}u^\alpha(x)},
\end{align*}
where for the sake of readability we shall usually stick to the summation convention on repeated indices and do not write out sums explicitly. We shall refer to $\lambda, \kappa $ and an upper bound for the $\L^\infty$-norm of $A$ as \emph{ellipticity} and to $n$ and the number $m \geq 1$ of equations as \emph{dimensions}.

Let us remark that for the local results we are after, it is no restriction to define $A$ on all of $\ree$. Indeed, if, for some open interval $I_{0}\subset \R$ and ball $Q_{0}\subset \R^n$, coefficients $A\in \L^\infty(I_0 \times Q_0;\cL(\IC^{mn}))$ satisfy \eqref{eq:accrassumption} only for $u \in \W_0^{1,2}(Q_0; \IC^m)$ uniformly in $t \in I_0$, then given $\eps \in (0,1)$, there exists $\widetilde A$ with $\widetilde A = A$ on $ (1-\eps)^2I_0 \times (1-\eps) Q_0 $ that satisfies \eqref{eq:accrassumption}. The ellipticity constants for $\widetilde A$ are possibly different and may depend on $\eps$, see Lemma~\ref{lem::extendGarding} in the appendix.

\subsection{Weak solutions}

Let $I_{0}$ be an open interval, $Q_{0}$ be an open ball of $\R^n$ and $\Omega := I_{0} \times Q_{0}$. We denote by $\ell(I_0)$ the length of $I_0$ and by $r(Q_0)$ the radius of $Q_{0}$. Given $f\in \Lloc^{1}(\Omega; \IC^{m})$ and $F\in \Lloc^{1}(\Omega; \IC^{mn}) $, we say that $u$ is a \emph{weak solution} to 
\begin{align*}
 \partial_t u -\div_{x} A(t,x)\nabla_{x} u = f+\div_x F
\end{align*}
in $\Omega$ if $u\in \Lloc^2(I_{0}, \W_{\loc}^{1,2}(Q_{0}; \IC^m))$ and for all smooth functions with compact support $\phi \in \C^\infty_{0}(\Omega;\IC^m)$,
\begin{align}
\label{eq:system}
\begin{split}
  \iint_{\Omega} A(t,x)\nabla_{x} u(t,x)\cdot\overline{\nabla_{x} \phi(t,x)} \; \d x \, \d t - \iint_{\Omega} u(t,x) \cdot \cl{\partial_{t}\phi(t,x)} \; \d x \, \d t \\  =  \iint_{\Omega} f(t,x) \cdot \cl{\phi(t,x)} \; \d x \, \d t -  \iint_{\Omega} F(t,x)\cdot\overline{\nabla_{x} \phi(t,x)} \; \d x \, \d t.
\end{split}
\end{align}
Here, $F \cdot \overline{\nabla_x \phi}$ is short for $F^{\alpha,i} \overline{\partial_i \phi^\alpha}$.

Having posed the setup, whenever the context is clear we are going to ignore the target spaces $\IC^m$ or $\IC^{mn}$ in our notation and do not write the Lebesgue measures $\d x$ and $\d t$. We abbreviate $\nabla:= \nabla_x$ and $\div:= \div_{x}$ for the gradient and divergence in the spatial variables $x$, respectively.

\subsection{Fractional time derivatives and related spaces}

In the following $\dhalf$ and $\HT$ denote the \emph{half-order time derivative} and \emph{Hilbert transform in time} defined on $\cS'(\R)/\IC$, the tempered distributions modulo constants, through the Fourier symbols $|\tau|^{1/2}$ and  $\i \sgn(\tau)$, respectively. For summarizing properties see for example Section~3 in \cite{AEN}. In particular, the time derivative factorizes as $\partial_t = \dhalf \HT \dhalf$. 

We shall use the space $\H^{1/2}(\R; \L^2(\R^n))$ of functions in $\L^2(\ree)$ such that $\dhalf f\in \L^2(\ree)$. Here, we identify $\L^2(\ree)$ with $\L^2(\R; \L^2(\R^n))$, and having said this, we extend $\dhalf$ and $\HT$ to $\ree$ by acting only on the time variable. 

More generally, for $1<p<\infty$ we introduce the spaces $\H^{1/2,p}(\R; \L^p(\R^n))$ of functions in $\L^p(\ree)$ such that $\dhalf f\in \L^p(\ree)$ with norm $(\|f\|_{p}^p+ \|\dhalf f\|_{p}^p)^{1/p}$. For the sake of completeness only, we remark that up to equivalent norms these are the (vector-valued) Bessel potential spaces usually denoted by the same symbol \cite{BL}. We also note that $\C^\infty_{0}(\ree)$, the space of smooth and compactly supported functions, is dense in these spaces using smooth convolution and truncation. Lebesgue space norms are denoted with the usual symbol $\| \cdot \|_{p}$.
\section{The global variational setup}
\label{sec:i}

Our starting point is the parabolic problem on the whole space $\ree$. We define the Hilbert space 
\begin{align*}
 V :=\L^2(\R; \W^{1,2}(\R^n))
\end{align*}
with norm $\|u\|_{V}:= (\|u\|_{2}^2 + \|\nabla u \|_{2}^2)^{1/2}$. This is the natural space for global weak solutions to parabolic problems. We recall that $2^*=2+\frac{4}{n}$ is the upper Sobolev conjugate of $2$ and that $(2^*)' = 2_*$ is its dual exponent. 

The following proposition summarizes the properties of global weak solutions.

\begin{prop}\label{prop:1} Let $f\in \L^{2_*}(\ree)$ and $F\in \L^2(\ree)$. Consider    $v\in V$ a weak solution to 
  $\partial_t v -\div A(t,x)\nabla v =  f+ \div  F$ in $\ree$. Then
  \begin{enumerate}
  \item $v\in \H^{1/2}(\R; \L^2(\R^n))$, 
  \item $v\in \L^{2^*}(\ree)$, 
  \item $v\in \C_{0}(\R; \L^2(\R^n))$ and $t\mapsto \|v(t,\blank)\|_{2}^2$ is absolutely continuous on $\R$,
\end{enumerate}
with 
$$
\sup_{t\in \R} \|v(t,\blank)\|_{2}+ \|v\|_{2^*}+ \|\dhalf v\|_{2}\lesssim \|v\|_{V}+ \|f\|_{2_*}+   \|F\|_{2}.
$$
The implicit constant depends only on dimensions and ellipticity. 
\end{prop}

We need a few short lemmas to prepare its proof. The key tool is the following definition of the parabolic operator through a sesquilinear form and its accretivity on the \emph{parabolic energy space} 
\begin{align*}
 E:= V\cap \H^{1/2}(\R; \L^2(\R^n))
\end{align*}
with norm $\|u\|_{E} := (\|u\|_{2}^2 + \|\nabla u \|_{2}^2+ \|\dhalf u\|_{2}^2)^{1/2}$. The result is basically that of~\cite{Kaplan, HL} but we repeat the short argument for the reader's convenience.

\begin{lem}
\label{lem:2}  
The operator $\cL:= \partial_t -\div A(t,x)\nabla + \kappa + 1$ can be defined as a bounded operator from $E$ to its dual $E^*$ via
\begin{align*}
 \langle \cL u, v \rangle := \iint_\ree A \nabla u \cdot \cl{\nabla v} + \HT \dhalf u \cdot \cl{\dhalf v} + (\kappa + 1) u \cdot \cl{v} \; \d x \, \d t, \qquad u,v \in E.
\end{align*}
This operator is invertible and its norm as well as the norm of the inverse depend only on ellipticity and dimensions.
\end{lem}

\begin{proof} 
The $E \to E^*$ boundedness of $\cL$ is clear by definition. Next, for the invertibility, the form 
\begin{align*}
 a_\delta(u,v) := \iint_{\ree} A \nabla u \cdot \cl{\nabla (1+\delta \HT) v} + \HT \dhalf u \cdot \cl{\dhalf (1+\delta \HT) v} + (\kappa + 1)u \cdot \cl{(1+\delta \HT) v} \; \d x \, \d t,
\end{align*}
for $u,v\in E$, is bounded and satisfies an accretivity bound for $\delta>0$ sufficiently small. Indeed, from the ellipticity condition and the fact that the Hilbert transform is $\L^2$-isometric,   skew-adjoint and commutes with $\dhalf$ and $\nabla$,  
\begin{equation*}\label{eq:coer}
\Re a_\delta(u,u) \ge (\lambda-\|A\|_{\infty}\delta )\|\nabla u\|_2^2 + \delta \|\dhalf u \|_2^2  + \|u\|_{2}^2. \end{equation*}
 As  
$$
\dual {\cL u}{ (1+\delta\HT)v} = a_\delta(u,v), \qquad u,v\in E, 
$$
and since $(1+\delta^2)^{-1/2}(1+\delta\HT)$ is isometric on $E$ as is seen using its symbol $(1+\delta^2)^{-1/2}(1+ \i \delta \sgn\tau )$, it follows from the Lax-Milgram lemma that $\cL$ is invertible from $E$ onto $E^*$.
\end{proof}

The following lemma is well-known on the Hilbert space $V$, see \cite[Prop.~III.1.2]{Showalter}, but we need a slight variant involving the smaller spaces 
\begin{align*}
 V_p:= V\cap \L^{p}(\ree), \quad 1<p<\infty,
\end{align*}
which are complete for the norm $\|v\|_{V_{p}}=\max(\|v\|_{V}, \|v\|_{p})$. Of course, we have $V_2 = V$. Since $\C_0^\infty(\ree)$ is dense in both $V$ and $\L^{p}(\ree)$ through approximation via smooth convolution and truncation, we can identify their larger duals $V_p^*$ (compared to $V^*)$ to the sum $V^* + \L^{p'}(\ree) = \L^2(\R; \W^{-1,2}(\R^n)) + \L^{p'}(\ree)$ and the $V_p^*$-$V_p$ duality can be realized as a Lebesgue integral
\begin{align*}
 \langle v^*, v \rangle = \int_{\R} \langle v_1^*(t,\blank), v(t,\blank) \rangle_{\W^{-1,2}, \W^{1,2}} + \langle v_2^*(t,\blank), v(t,\blank) \rangle_{\L^{p'}, \L^p} \; \d t,
\end{align*}
where $v^* = v_1^* + v_2^*$ is any admissible decomposition, see Theorem~2.7.1 in \cite{BL}.

\begin{lem}
\label{lem:3}  
Let $1<p<\infty$ and consider a function $v\in V_{p}$ such that $\pd_{t}v\in  V_{p}^*$. Then $v \in \C(\R; \L^2(\R^n))$ and $t\mapsto \|v(t,\blank)\|_{2}^2$ is absolutely continuous on $\R$, vanishes at $\pm\infty$ and satisfies
\begin{align*}
 \sup_{t\in \R}\|v(t, \blank)\|_{2}^2\le 2  \|v\|_{V_{p}}\|\pd_{t}v\|_{V_{p}^*}.
\end{align*}
Moreover, it holds $\Re \langle \pd_{t}v, v \rangle = 0$ for the $V_p^*$-$V_p$ duality. 
\end{lem}

\begin{proof}
We approximate $v$ through convolution with smooth compactly supported kernels in the $t$-variable only, say $v_\eps = v \ast_t \varphi_\eps$, where $\eps >0$. By construction $v_\eps$ is in the class $\C^\infty(\R; \L^2(\R^n) \cap \L^p(\R^n))$ and vanishes at $\pm \infty$. Hence we can differentiate in the classical sense
\begin{align*}
\frac{\d}{\d t} \|v_\eps(t,\blank)\|_2^2 = 2 \Re \int_{\R^n} \partial_t v_\eps(t,x) \cl{v_\eps(t,x)} \; \d x.
\end{align*}
Integrating this identity over an interval $I$, we obtain
\begin{align}
\label{eq:dual1}
 \int_I \frac{\d}{\d t} \|v_\eps(t, \blank)\|_2^2 \; \d t
 = 2 \Re \iint_{\R^{n+1}} \partial_t v_\eps(t,x) \cl{1_{I}(t) v_\eps(t,x)} \; \d x \, \d t = 2 \Re \langle \partial_t v_\eps, 1_{I}v_\eps \rangle.
\end{align}
In the limit $\eps \to 0$ we have, by construction, $v_\eps \to v$ strongly in $V_p$. As for $\partial_t v_\eps$, we have boundedness in $V_p^*$ uniformly in $\eps$ and weak${}^*$-convergence towards $\partial_t v$. Hence, the right-hand side in \eqref{eq:dual1} converges as $\eps \to 0$ and we need to determine the limit of the left-hand side.

Let us take $I=(-\infty, T]$. First, we write the same equality \eqref{eq:dual1} for differences $v_\eps - v_{\eps'}$ and apply the fundamental theorem of calculus on the left-hand side. This reveals that $(v_\eps(T,\blank))_\eps$ is a Cauchy sequence in $\L^2(\R^n)$, uniformly in $T \in \R$. The approximants $v_\eps: \R \to \L^2(\R^n)$ are continuous. Thus, the limit $v:\R \to \L^2(\R^n)$ is also continuous. Now we pass to the limit as $\eps \to 0$ in \eqref{eq:dual1}. The left-hand side tends to $\|v(T, \blank)\|_{2}^2$ whereas the right-hand side tends to $2 \Re \langle \partial_t v, 1_{I} v \rangle \leq 2 \|\pd_{t}v\|_{V_{p}^*} \|v\|_{V_{p}}$.

Next, we take an arbitrary bounded interval  $I = (a,b)$, pass again to the limit as $\eps \to 0$ in \eqref{eq:dual1} and write out the duality explicitly. This yields
\begin{align*}
 \|v(b,\blank)\|_2^2 - \|v(a,\blank)\|_2^2 = \int_{a}^b \langle f(t,\blank), v(t,\blank) \rangle_{\W^{-1,2}, \W^{1,2}} + \langle g(t,\blank), v(t,\blank) \rangle_{\L^{p'}, \L^p} \; \d t,
\end{align*}
where $\partial_t v = f + g$ with $f \in \L^2(\R; \W^{-1,2}(\R^n))$ and $g \in \L^{p'}(\ree)$. The integrand on the right is in $\L^1(\R)$ by H\"older's inequality. Hence, $t\mapsto \|v(t,\blank)\|_{2}^2$ is absolutely continuous on $\R$.

We may obtain the final statement $\Re \langle \pd_{t}v, v \rangle = 0$ by taking $I = \mathbb{R}$ in \eqref{eq:dual1} and passing to the limit as $\eps \to 0$.   
\end{proof}

Finally, we need a slight variant of the usual parabolic Sobolev embeddings. For background, we refer to~\cite{SobPara}.

\begin{lem}\label{lem:1} Let $1 < p < n+2$. Then for all $\phi \in \C^\infty_{0}(\ree)$,
$$
\|\phi\|_{p^*}\lesssim \|\nabla \phi\|_{p}+\|\dhalf \phi\|_{p}.
$$ 
\end{lem}

\begin{proof} 
Let $\cF$ be the Fourier transform on $\ree$ and let $(\tau,\xi)$ be the Fourier variable corresponding to $(t,x)$. The Sobolev inequality in parabolic scaling from \cite{SobPara} gives $\|\phi\|_{p^*} \lesssim \|\cF^{-1}((\i \tau + |\xi|^2)^{1/2} \cF \phi)\|_p$. So, in order to conclude, it suffices to remark that the operators defined on the Fourier side by multiplication with $(\i \tau + |\xi|^2)^{1/2}/(|\tau|^{1/2} + |\xi|)$ {and $\xi/|\xi|$} are bounded on $\L^p(\ree)$ by the Marcinkiewicz multiplier theorem, see Corollary~5.2.5 in \cite{Grafakos}.
\end{proof}

We can now give the

\begin{proof}[Proof of Proposition~\ref{prop:1}] 
In virtue of the canonical identifications we have the continuous inclusion $E \subset E^*$ and a bounded mapping $\div: \L^2(\ree) \to E^*$. It follows from Lemma~\ref{lem:1} and density of $\C^\infty_{0}(\ree)$ in $E$ (standard mollification and truncation), that $E$ embeds into $\L^{2^*}(\ree)$. Hence, $\L^{2_*}(\ree)$ embeds into $E^*$.  Thus, we can state $f+(\kappa + 1)v+ \div F\in E^*$. 

It follows from Lemma~\ref{lem:2} that there exists $\widetilde v\in E$ such that $\cL \widetilde v =f+(\kappa + 1)v+\div F$. By definition of the respective embeddings, this means that for all $\phi \in E$,
\begin{align*}
 \iint_\ree A \nabla \widetilde v \cdot \cl{\nabla \phi} + \HT \dhalf \widetilde v \cdot \cl{\dhalf \phi} + (\kappa + 1)\widetilde v \cdot \cl{\phi} \; \d x \, \d t = \iint_\ree (f+(\kappa + 1)v) \cl{\phi} - F \cdot \cl{\nabla \phi} \; \d x \, \d t.
\end{align*}
Restricting to $\phi \in \C_0^\infty(\ree)$, we can write $\partial_t \phi = \dhalf \HT \dhalf \phi$ and see in particular that $u:=v-\widetilde v\in V$ is a weak solution to $\partial_t u -\div A(t,x)\nabla u +(\kappa + 1)u=0$ in $\ree$. We may now apply the Caccioppoli inequality,
\begin{align}\label{eq:cac}
 \iint_{I \times Q} |\nabla u|^2\; \d x \, \d t  \lesssim \frac{1}{r(Q)^2} \iint_{4 I \times 2Q} |u|^2\; \d x \, \d t,
\end{align}
for any parabolic cylinder $I \times Q$ with $\ell(I) = r(Q)^2$, see Remark \ref{rem:Cac} below for convenience. Since we have $u\in \L^2(\ree)$, we obtain $\nabla u = 0$ on passing to the limit $r(Q) \to \infty$. Hence, $u$ depends only on $t$. Again, as $u\in \L^2(\ree)$, $u$ must be $0$.
It follows that $v=\widetilde v\in E$, hence (i) is proved and 
\begin{align*}
 \|\dhalf v\|_{2}\lesssim \|\cL v\|_{E^*} \leq (\kappa + 1)\|v\|_2+ \|f\|_{2_*}+ \|F\|_{2}
\end{align*}
follows by Lemma~\ref{lem:2}. Applying Lemma~\ref{lem:1} again yields (ii) by density. As for (iii) we have seen $v\in V_{2^*}$, which in turn implies $\pd_{t}v\in V_{2^*}^*$ by the equation for $v$. Hence, we can apply Lemma~\ref{lem:3} to $v$ and obtain the statements on continuity. We also obtain $\sup_{t\in \R}\|v(t, \blank)\|_{2}^2\le 2  \|v\|_{V_{2}}\|\pd_{t}v\|_{V_{2}^*}$ and so the required bound follows on controlling the right-hand side by means of (i) and (ii).
\end{proof}
 
\section{Local estimates}
\label{sec:localEstimates}

As a first application of the global results obtained in the previous section we present the ``classical'' local estimates for weak solution within our setting. We recall that by definition a weak solution $u$ to a parabolic problem in a parabolic cylinder $\Omega = I_0 \times Q_0$ satisfies $u \in \Lloc^2(\Omega)$ and $\nabla u \in \Lloc^2(\Omega)$.

The following lemma is nothing but a simple calculation. Nevertheless, it is of fundamental importance for all subsequent considerations. Here, we suggestively use the notation for the scalar case $m=1$ (even when $m>1$), as we are only interested in norm estimates later on.

\begin{lem}
\label{lem:localization}
Let $u$ be a weak solution in $\Omega$ to $\partial_t u -\div A(t,x)\nabla u  = f+\div F$ with $f \in \Lloc^1(\Omega)$ and $F \in \Lloc^{1}(\Omega)$. Let $\chi \in \C_0^\infty(\Omega; \IC)$ and put $v := \chi u$. Then $v \in V$ is a weak solution to 
\begin{align*}
 \partial_t v -\div A(t,x)\nabla v  = \widetilde f+\div \widetilde F 
\end{align*}
in $\ree$ with
\begin{align}
\label{eq:tilde}
\begin{split}
   \widetilde f &  = \chi f + (\pd_{t}\chi) u  - A\nabla u\cdot \nabla \chi - F\cdot \nabla \chi ,    \\
   \widetilde F &  = - A (u\nabla \chi) + F\chi.
\end{split}
\end{align}
\end{lem}

With this at hand we can prove the local higher integrability and absolute continuity in time of weak solutions.

\begin{thm}
\label{thm:reg-i}
Assume that $u$ is a weak solution to $\partial_t u -\div A(t,x)\nabla u  = f+\div F$ on $\Omega= I_{0} \times Q_{0}$ with right-hand side $f\in \Lloc^{2_*}(\Omega; \IC^m)$ and $ F\in \Lloc^2(\Omega; \IC^{mn})$. It holds 
\begin{align*}
 u\in \Lloc^{2^*}(\Omega; \IC^m) \cap \C(I_{0}; \Lloc^2(Q_{0}; \IC^m)).
\end{align*}
More precisely, for every $\chi \in \C_0^\infty(\Omega)$ the function $t \mapsto \|(u\chi)(t,\blank)\|^2$ is absolutely continuous on $\R$, it holds $\dhalf (u\chi)\in \L^2(\ree; \IC^m)$ and
\begin{equation}
\label{eq:qestimate}
\|u\chi\|_{2^*}\lesssim \| \widetilde f\|_{2_*}+ \|\widetilde F\|_{2}+ \|u\chi\|_{2}+ \|\nabla(u\chi)\|_{2}. 
\end{equation}
\end{thm}

\begin{proof}
Let $\chi\in \C_{0}^\infty(\Omega)$.  Set $v:=u\chi$. From Lemma~\ref{lem:localization} we know that $v\in V$ is a weak solution to $\partial_t v -\div A(t,x)\nabla v  = \widetilde f+\div \widetilde F$ in $\ree$ with $\widetilde f, \widetilde F$ given by \eqref{eq:tilde}. Using the assumption on $f, F$, the local square-integrability of $u$ and $\nabla u$, and $2_* < 2$, we see that  $\widetilde f \in \L^{2_*}(\ree)$ and $\widetilde F\in \L^2(\ree)$. The conclusion follows from Proposition~\ref{prop:1}.
\end{proof}

We continue with the Caccioppoli inequality. It will be convenient to formulate it with an additional zero-order term on the right-hand side.

\begin{prop}[Caccioppoli inequality]
\label{ref:cac} 
Let $u$ be a weak solution in $\Omega = I_0 \times Q_0$ to the parabolic problem
\begin{align*}
 \partial_t u - \div A(t,x) \nabla u = f + \div F - Bu,
\end{align*}
where $f\in \Lloc^{2_*}(\Omega)$, $F\in \Lloc^2(\Omega)$ and $B \in \L^\infty ( \Omega; \Lop(\IC^m))$ satisfies $\Re \int_{Q_0} B(x) \phi(x) \cdot \cl{\phi(x)} \; \d x \geq 0$ for all $\phi \in \W_0^{1,2}(Q_0; \IC^m)$.  Let $I \times Q \subset \Omega$ be open parabolic cylinder with $\ell(I) \sim r(Q)^2$ such that for some $\gamma>1$ also the closed cylinder $\overline{\gamma^2I\times \gamma Q}$ is contained in $\Omega$. Then
\begin{align}
\label{eq:Cac}
  \left(\bariint_{I \times Q} |\nabla u|^2  \right)^{1/2} \lesssim \frac{1}{r(Q)} \left(\bariint_{\gamma^2I\times \gamma Q} |u|^2\right)^{1/2} + \left(\bariint_{\gamma^2I\times \gamma Q} |F|^2\right)^{1/2} + r(Q) \left(\bariint_{\gamma^2I\times \gamma Q} |f|^{2_*}\right)^{1/2_*}.
\end{align}
The implicit constant depends on ellipticity, dimensions, $\gamma$, constants controlling the ratio $r^2/\ell$  and $\|B\|_\infty$.
\end{prop}

Before we give the proof of Caccioppoli's inequality, let us conclude the reverse H\"older estimate of Ne\v cas--\v Sver\'ak~\cite{Necas-Sverak} for $|u|$ in our setting.

\begin{prop}[Reverse H\"older estimate for $u$]
\label{ref:basic}   
Assume that $u$ is a weak solution to $\partial_t u -\div A(t,x)\nabla u  = f+\div F$ on $\Omega= I_{0} \times Q_{0}$ with right-hand side $f\in \Lloc^{2_*}(\Omega; \IC^m)$ and $ F\in \Lloc^2(\Omega; \IC^{mn})$. Let $I \times Q \subset \Omega$ be open parabolic cylinder with $\ell(I) \sim r(Q)^2$ such that for some $\gamma>1$ also the closed cylinder $\overline{\gamma^2I\times \gamma Q}$ is contained in $\Omega$. Then
\begin{align}
\label{eq:rhu}
   \left(\bariint_{I \times Q} | u|^{2^*}\right)^{1/2^*}  \lesssim    \bariint_{\gamma^2I\times \gamma Q} |u| + r(Q) \left(\bariint_{\gamma^2I\times \gamma Q} |F|^2\right)^{1/2} + r(Q)^2 \left(\bariint_{\gamma^2I\times \gamma Q} |f|^{2_*}\right)^{1/2_*},
\end{align}
where the implicit constants depend only on ellipticity, dimensions, $\gamma$ and the constants controlling the ratio $r^2/\ell$. 
\end{prop}

\begin{proof} 
The equation  \eqref{eq:rhu} with $\big(\bariint_{\gamma^2I\times \gamma Q} |u|^2\big)^{1/2}$ on the right-hand side follows from \eqref{eq:qestimate} and Proposition~\ref{ref:cac} -- at least when $r =1$, which suffices since our hypotheses are invariant under rescaling.  The improvement to $\bariint_{\gamma^2I\times \gamma Q} |u|$ follows from a classical self-improvement feature of reverse H\"older inequalities, see Theorem~2 in \cite{IN}. A simple proof that applies in our situation with parabolic scaling can be found in Theorem~B.1 of \cite{BCF}.
\end{proof}

\begin{rem}
Under suitable assumptions on $f$ and $F$ the classical Gehring lemma (with parabolic scaling) can be used to improve the exponent of integrability on the left-hand side to $2^* + \eps$, where $\eps>0$ depends on ellipticity, dimensions, $\gamma$ and the constants controlling the ratio $r^2/\ell$. For $f=0$ and $F=0$, the argument is found in the textbook~\cite{Bjorn-Bjorn}.
\end{rem}

Concatenating \eqref{eq:Cac} and \eqref{eq:rhu} yields

\begin{cor}[Improved Caccioppoli inequality]
Under the assumptions and with the notation of Proposition~\ref{ref:basic}, it holds
\begin{align}
\label{eq:Cac1}
  \left(\bariint_{I \times Q} |\nabla u|^2  \right)^{1/2} \lesssim \frac{1}{r(Q)} \bariint_{\gamma^2I\times \gamma Q} |u| + \left(\bariint_{\gamma^2I\times \gamma Q} |F|^2\right)^{1/2} + r(Q) \left(\bariint_{\gamma^2I\times \gamma Q} |f|^{2_*}\right)^{1/2_*}.
\end{align}
\end{cor}

We complete the section with the proof of Caccioppoli's inequality. The argument follows the traditional one and can be omitted on a first reading.

\begin{proof}[Proof of Proposition~\ref{ref:cac}]
For the argument set $q:=2^*$ with H\"older conjugate $q'=2_*$. By scaling we may assume $r=1$ as before. We pick $\chi\in \C_{0}^\infty(\ree)$, real-valued, with $\chi=1$ on $I \times Q$ and support contained in $\gamma^2 I \times \gamma Q$.  As in Lemma~\ref{lem:localization} we write the equation satisfied by $v:=u\chi$ as $\partial_t v   = \widetilde f+\div (A \nabla v + \widetilde F) - Bv$ with 
\begin{align}
\label{eq:tilde_again}
\begin{split}
   \widetilde f &  = \chi f + (\pd_{t}\chi) u  - A\nabla u\cdot \nabla \chi - F\cdot \nabla \chi,\\
   \widetilde F &  = - A (u\nabla \chi) + F\chi.
\end{split}
\end{align}
We have $v \in V_q$ thanks to Proposition~\ref{prop:1} and from the assumptions on $f$, $F$, $B$ and H\"older's inequality we can infer $\partial_t v \in V_q^*$. Thus, Lemma~\ref{lem:3} yields
 \begin{equation}
\label{eq:cac1}
 0= \Re \langle \partial_t v, v \rangle = -  2 \Re \iint (A\nabla v+\widetilde F)\cdot \overline{\nabla v} \; \d x \, \d t+   2 \Re \iint \widetilde f\,  \overline{ v}\; \d x \, \d t - 2 \Re \iint Bv \cdot \cl{v} \; \d x \, \d t.
\end{equation}
First, we note that the final integral on the right has positive real part by assumption. We then isolate $\Re (A\nabla v \cdot \overline{\nabla v})$ and use Young's inequality to give
\begin{align*}
 2  \Re \iint A\nabla v\cdot \overline{\nabla v} 
 \le     \iint  \lambda   |\nabla v|^2 + \lambda^{-1}     |\widetilde F  |^2 + 2  |\widetilde f v|.
\end{align*}
Now, we apply G\aa rding's inequality \eqref{eq:accrassumption} on the left and cancel $ \lambda |\nabla v|^2$ on both sides to obtain
\begin{align}
\label{eq:cac2}
 \lambda  \iint |\nabla v|^2 \le   \iint \lambda^{-1} |\widetilde F  |^2 + 2 |\widetilde f v| + 2\kappa|v|^2. 
\end{align}
We pick yet another real-valued function $\phi\in \C_{0}^\infty(\ree)$, supported in $\gamma^2 I \times \gamma Q$ with $\phi = 1$ on the support of $\chi$. On recalling $v = \chi u$, we can hence write 
\begin{align*}
  \widetilde f v&  = f\phi u \chi^2 + (\pd_{t}\chi) \chi u^2  - A\nabla u\cdot (\nabla \chi) u\chi - F\cdot (\nabla \chi) u \chi.
\end{align*}
We insert this expression along with the definitions of $v$ and $\widetilde F$ into \eqref{eq:cac2}. Then we can estimate all terms appearing on the right but $f\phi u \chi^2$ simply by Cauchy-Schwarz' and Young's inequality and use the uniform bound for $A$ whenever convenient. This results in
\begin{align}
\label{eq:cac3}
  \lambda  \iint |\nabla (u\chi)|^2 
  \lesssim  \text{I} + \iint |u\chi A \nabla u \cdot \nabla \chi| + \iint |f\phi||u\chi^2|,
\end{align}
where
\begin{align*}
 \text{I} := \iint |u|^2 ( |\nabla \chi|^2 +|\pd_{t}\chi|+ |\chi|^2) + \iint |F|^2(|  \chi|^2+ |\nabla \chi|^2).
\end{align*}
As for the second term on the right-hand side of \eqref{eq:cac3}, we write
\begin{align*}
 u \chi A \nabla u \cdot \nabla \chi = u \Big(A \nabla (u\chi) \cdot \nabla \chi - A (u \nabla \chi) \cdot \nabla \chi \Big).
\end{align*}
This allows us to apply Young's inequality with parameters chosen such that the contribution of $\nabla (u\chi)$ appearing on the right-hand side of \eqref{eq:cac3} can be absorbed into the left-hand side. Invoking again the uniform bound for $A$, all other terms created in this step will only increase the implicit constant in front of $\text{I}$. Thus, we can note
\begin{align}
\label{eq:cacciointermediate}
  \iint |\nabla (u\chi)|^2  \lesssim  \text{I } + \iint |f\phi||u\chi^2|.
\end{align}
Had we assumed $f \in \Lloc^2(\Omega)$, then a simple application of the Cauchy-Schwarz inequality would complete the argument with an $\L^2$ average on the right-hand side. But we only assumed $f \in \Lloc^{q'}(\Omega)$. 

In order to master the situation, we introduce $w = u \chi^2$, which is of the same nature as $v$ except that the cut-off function changed from $\chi$ to $\chi^2$. We also define $f', F'$ just as $\widetilde f, \widetilde F$ but with $\chi$ replaced by $\chi^2$. H\"older's inequality followed by the global bound provided by Proposition~\ref{prop:1} (iii) for $w$ yields
\begin{align}
\label{eq:cac4}
\iint |f\phi||u\chi^2| 
\leq \|f\phi\|_{q'} \|w\|_{q}
&\lesssim \|f\phi\|_{q'} \big( \|f'\|_{q'}+ \|F'\|_{2}+ \|w\|_{2}+ \|\nabla w\|_{2} \big).
\end{align}
Crudely using H\"older's inequality on the definition of $f'$, we find
\begin{align*}
 \|f'\|_{q'} \leq \|f \chi \|_{q'} + \|\pd_{t}\chi u\|_2  + \|A\nabla u\cdot \nabla \chi^2\|_2 + \|F\cdot \nabla \chi\|_2.
\end{align*}
We see that up to changing the cut-off function from $\chi$ to $\chi^2$, the terms $\|F'\|_{2}$ and $\|w\|_{2}$ on the right of \eqref{eq:cac4} have already been estimated before when passing from \eqref{eq:cac2} to \eqref{eq:cac3}. Repeating these arguments,
\begin{align*}
 \iint |f\phi||u\chi^2| \leq \text{I } + \|f \phi\|_{q'}^2 + \|f \phi \|_{q'} \|A\nabla u\cdot \nabla (\chi^2)\|_2 + \|f \phi\|_{q'} \|\nabla w\|_2.
\end{align*}
Since we have
\begin{align*}
 A \nabla u \cdot \nabla(\chi^2) = 2 A \nabla (u\chi) \cdot \nabla \chi- 2 A(u \nabla \chi) \cdot \nabla \chi
\end{align*}
and
\begin{align*}
 \nabla w = \nabla(u \chi^2) = \chi \nabla (u\chi) + u\chi \nabla \chi,
\end{align*}
we can use Cauchy-Schwarz' and Young's inequality on the $\L^2$-norms of these two terms to give
\begin{align*} 
 \iint |f\phi||u\chi^2| \leq C(\eps) (\text{I } + \|f \phi\|_{q'}^2) + \eps \|\nabla(u \chi)\|_2^2.
\end{align*}
Here, $\eps > 0$ is at our disposal and $C(\eps)$ is a finite constant that depends on $\eps$. Picking $\eps>0$ small enough, this estimate together with \eqref{eq:cacciointermediate} leads to $\|\nabla(u\chi)\|_2^2 \lesssim \text{I } + \|f \phi\|_{q'}^2$. The conclusion follows from the definition of $\text{I}$ and the defining properties of $\chi$ and $\phi$.
\end{proof}

\begin{rem}
\label{rem:Cac}
Under the stronger assumption $f \in \Lloc^2(\Omega)$ the proof given above yields
\begin{align}
\label{eq:CacUsual}
  \left(\bariint_{I \times Q} |\nabla u|^2  \right)^{1/2} \lesssim \frac{1}{r(Q)} \left(\bariint_{\gamma^2I\times \gamma Q} |u|^2\right)^{1/2} + \left(\bariint_{\gamma^2I\times \gamma Q} |F|^2\right)^{1/2} + r(Q) \left(\bariint_{\gamma^2I\times \gamma Q} |f|^{2}\right)^{1/2}
\end{align}
\emph{without} making use of Proposition~\ref{prop:1}. This observation is important in order to make clear that there is no circular reasoning going on in the proof of the latter. Indeed, we can replace $q$ by $2$ so that $u \in V_q = V_2$ is by definition of weak solutions  and then we follow the proof \emph{verbatim} until we reach \eqref{eq:cacciointermediate}, where now we only have to apply the Cauchy-Schwarz inequality to conclude.
\end{rem}

\section{A Gehring type lemma with tail}
\label{sec:Gehring}
 
We provide here the main real analysis lemma to obtain our estimates. For a ball $Q \subset \mathbb{R}^{n}$ and an interval $I \subset \R$ with $\sqrt{\ell(I)} = r(Q) := r$ we write $B := I \times Q$. If $(t,x)$ is the center of $B$, we also use the notation $B = B((t,x),r)$ and $r = r(B)$ for such a parabolic cylinder (that is, a cylinder which is a  ball in the parabolic (quasi-)metric $d((t,x), (s,y)) := \max(\sqrt{2|t-s|}, |x-y| )$). For $u\ge 0$ locally integrable and $\gamma>1$ we define $a_u(B) \in [0,\infty]$ through
\begin{align*}
 a_u (B) := \sum_{j = 0}^{\infty} 2^{-j-1} \bariint_{4^{j} I \times \gamma Q} u \; \d x \, \d t :=  \sum_{j = 0}^{\infty} 2^{-j-1} \barint_{4^{j} I \times \gamma Q} u\, d\mu,
\end{align*}
where for this section $\mu$ denotes the Lebesgue measure on $\ree$ and we use the single integral notation for simplicity. The functional $a_u$ is an approximate identity indexed over radii of parabolic cylinders  when $u \in \L^{p}(\ree)$ for some $p \in (1,\infty)$ in the sense that $a_u(B((t,x),r)) \to u(t,x)$ as $r \to 0$ for almost every $(t,x)$. Indeed, introduce  the maximal operators $M_{x}$ and $M_{t}$  on space and time variables separately. For each $j\ge 0$ we have
\begin{align*}
 \bigg|\barint_{4^{j} I \times \gamma Q} u\, d\mu \bigg| \le M_t M_x u(t,x)
\end{align*}
and as $M_{t}M_{x}$ is bounded on $\L^p(\ree)$, this average converges to $u(t,x)$ for almost every point. So, the claim follows from the dominated convergence theorem for series and $ \sum_{j = 0}^{\infty} 2^{-j-1}=1$.  In addition,  we have $a_u(B((t,x),r)) \to 0$ when $r \to \infty$ by H\"older's inequality. {This last point also holds when $u \in \L^{1}(\ree)$}.

\begin{lem}\label{lem:global}
Let $g,f,h$ be be non-negative functions with  $g^2, f ^{2} , h^{s}  \in \L^{1}(\ree)$ for some $1<s <n+2$,  and suppose that for some $A\ge 1$,
\begin{equation*}
\left( \barint_{B} g^{2}\, d\mu \right)^{1/2} \le A a_g(B) + (a_{f^{2}}(B))^{1/2} + r(B) (a_{h^{s}} (B) )^{1/s}
\end{equation*} 
holds for all parabolic cylinders $B$. Let $p>2$ and suppose there are $\alpha, \beta \geq 0$ and $q_\alpha, q_\beta > 1$ (depending on $p$) such that
\begin{align}
\label{eq:a,b,qa,qb}
 2\alpha + \beta = s \qquad  \text{and} \qquad \frac{1}{q_\alpha} - \alpha = \frac{s}{p} = \frac{1}{q_\beta} - \frac{\beta}{n}.
\end{align}
If $|p-2|$ is sufficiently small (depending on $A$ and dimension), then
\[ \|g\|_{\L^{p}(\mathbb{R}^{n+1})} \lesssim  \|f\|_{\L^{p}(\mathbb{R}^{n+1})} + \|h^{s}\|_{\L^{q_\alpha}(\mathbb{R}; \L^{q_\beta}(\mathbb{R}^n))} ^{1/s}. \]
The implicit constant depends on $A$, $\alpha$, $\beta$, $q_\alpha$, $q_\beta$ and dimension.
\end{lem}

\begin{proof}
Let $m  > 0$. Denote $g_m := \min (g,m)$. By the Cavalieri principle we have 
\begin{align*}
\int_\ree g_m^{p-2}g^2\, d\mu  & = (p-2)\int_0^m \lambda^{p -2 - 1} g^2(\{g > \lambda\})\, d\lambda,
\end{align*}
where $g^2(A):=\int_{A}g^2\, d\mu$.
We define three functions 
\begin{align*}
G(t,x,r)	&: =  a_g(B((t,x),r)), \\
F(t,x,r)	&: = a_{f^2}(B((t,x),r))^{1/2}, \\
H(t,x,r)	&: = r \cdot a_{h^s}(B((t,x),r))^{1/s}
\end{align*}
and for $\lambda > 0$, we denote $E_\lambda :=  \{g > \lambda\}$. We have 
{\[\lim_{r \to 0} G(t,x,r) = g(t,x) \]}
for almost every $(t,x)$ by the discussion before the statement of the lemma and we define $\widetilde E_{\lambda}$ as the subset of $E_{\lambda}$ where this holds. We also note 
\[\lim_{r \to \infty} \big( G(t,x,r) + F(t,x,r) + H(t,x,r) \big) = 0 \]
for all $(t,x)$, using the global assumptions on $g,f,h$ and $s< n+2$.

By definition, if $(t,x) \in \widetilde E_\lambda$, then
\[\lim_{r \to 0} G(t,x,r)  > \lambda,\]
and thus for $(t,x) \in \widetilde E_\lambda$ we can define the stopping time radius
\begin{equation*}
r_{t,x}:= \sup\{ r > 0:  G(t,x, r) + F(t,x,r) + H(t,x,r) > \lambda \}.
\end{equation*}
We readily see that  $\sup_{(t,x)\in \widetilde E_{\lambda}} r_{t,x} < \infty$. Indeed, since $G, F, H$ are continuous functions of $r > 0$ for fixed $(t,x)$, we have at $r=r_{t,x}$ {equality $G(t,x, r) + F(t,x,r) + H(t,x,r) = \lambda$ and thus} either $G(t,x, r)\ge \lambda/3$  or  $F(t,x,r) \ge \lambda/3$ or $H(t,x,r) \ge \lambda/3$. In the last case for example, we obtain 
\begin{align*}
 r^{n+2-s}(\lambda/3)^s \lesssim \int_{\ree} h^s \, d\mu <\infty
\end{align*}
and the other cases give us an upper bound on $r$ in a similar manner. By the Vitali covering lemma, there exists an absolute constant $K$ and a countable collection of balls $\{B((t_{i},x_i),r_i)\} = \{B_i\}$ with $r_i = r_{t_{i},x_i}$ such that the $\tfrac{1}{K}B_i$ are pairwise disjoint and $\widetilde E_\lambda \subset \cup_i B_i$. (A value of $K$ can be computed explicitly by following the usual proofs in this particular quasi-metric.)

Now, using the hypothesis for each $B_{i}$ and pairwise disjointness of the balls $\tfrac{1}{K}B_i$, we find
\begin{align*}
g^2(\widetilde E_{\lambda})  &\le \sum_i  g^2({B_i}) 
		\le \sum_i \mu(B_i) \Big(Aa_g(B_i) + (a_{f^2}(B_i))^{1/2} + r_i (a_{h^s}(B_i))^{1/s}\Big)^2 \\
		&\leq A^2 \sum_i  \mu(B_i)\lambda^{2} \le A^2 K^{n+2} \sum_i  \mu(\tfrac 1 K B_i) \lambda^{2} \le A^2 K^{n+2} \lambda^{2}  \mu \bigg( \bigcup_{i}  B_i \bigg).
\end{align*}
Let $M_x^{\beta}$ be the fractional maximal function with respect to the $x$-variable:
\[M_x^{\beta}v(t,x):= \sup_{Q\ni x}   r(Q)^{ \beta } \barint_{Q} |v(t,y)| \,  d y.  \] 
Similarly, define $M_{t}^\alpha$ with respect to the $t$-variable. Since $2\alpha + \beta = s$, the parabolic scaling $r(B) = r(Q) = \sqrt{\ell(I)}$ yields $r(B)^s = r(Q)^\beta \times \ell(I)^\alpha$. Thus, it follows from the definition of $r_i$ that
\begin{align*}
 \bigcup _i B_i \subset \big\{M_t M_x g \ge \lambda/3\big\} \cup \big\{  M_t M_x(f^2) \ge (\lambda/3)^2 \big\} \cup  \big\{  M_t^{\alpha}M_x^{\beta } (h^s) \ge (\lambda/3)^s \big\}  =: S_\lambda.
\end{align*}

We thus have established
\[g^2(E_\lambda)=  g^2(\widetilde E_\lambda) \leq A^2 K^{n+2} \lambda^{2} \mu(S_\lambda). \]
Going back to the start of the proof, so far we have found
\begin{align}
\label{eq1:Gehring}
\begin{split}
 \int_\ree g_m^{p-2}g^2\, d\mu
&=(p-2)\int_0^m \lambda^{p - 3} g^2(\{g > \lambda\})\, d\lambda \\
&\leq A^2 K^{n+2} (p-2)\int_0^m \lambda^{p  - 1} \mu(S_\lambda)\, d\lambda  \\
&\leq A^2 K^{n+2}(\text{I} + \text{II} + \text{III}),
\end{split}
\end{align}
where the integrals $\text{I}, \text{II}, \text{III}$ correspond to the decomposition of $S_\lambda$ above. By the Cavalieri principle, we obtain for $p>2$,
\begin{align*}
\text{II} 		&\le \frac{p-2}{p} \int_\ree M_t M_x (f^2) ^ {p/2}\, d\mu  \lesssim \frac{p}{p-2}   \int_{\ree} f^{p}\, d\mu 
\end{align*}
by iterating the two maximal function $\L^{p/2}$ bounds, so that the implicit constant depends only on the dimension $n$. {Note that $p>2$ and that $p$ is determined by the other parameters in \eqref{eq:a,b,qa,qb}}. 
Similarly,
\begin{align*}
\text{III}		&\le \frac{p-2}{p} \int_\ree  M_t^{\alpha} M_x^{ \beta } (h^{s}) ^{p/s} \, d\mu .
\end{align*}
By hypothesis,  we have exponents $q_\alpha, q_\beta > 1$ such that
\[ \frac{1}{q_\alpha} - \alpha = \frac{s}{p} = \frac{1}{q_\beta} - \frac{\beta}{n}. \]
With a slight abuse in our notation, ignoring the other variable, these are precisely the conditions guaranteeing that $M_t^{\alpha} : \L^{q_\alpha}(\mathbb{R}) \to \L^{p/s}(\mathbb{R})$ and $M_x^{\beta} : \L^{q_\beta}(\mathbb{R}^{n}) \to \L^{p/s}(\mathbb{R}^{n})$ are bounded, see Theorem~3.1.4 in \cite{Adams-Hedberg}. Now, using this and Minkowski's inequality along with $s q_\alpha/ p = 1 - \alpha q_\alpha \leq 1$  in the second step, we see that
\begin{align*}
 \int_\ree   M_t^{\alpha}M_x^{ \beta }(h^{s}) ^{p/s}\, d\mu  
 &\lesssim   \int_{\R^n} \bigg(\int_{\R} M_x^{\beta}(h^s)^{q_\alpha} \, d t\bigg)^{p/(sq_\alpha)} \; \d x\\
 &\leq  \bigg(\int_{\R} \bigg(\int_{\R^n} M_x^{\beta}(h^s)^{p/s} \, d x\bigg)^{sq_\alpha/p} \; \d t \bigg)^{p/(sq_\alpha)}
 \lesssim \|h^{s}\|_{\L^{q_\alpha}(\mathbb{R};  \L^{q_\beta}(\mathbb{R}^n))}^{p/s}
\end{align*}
with implicit constant depending on $\alpha, \beta, q_\alpha, q_\beta$ and dimension.
The remaining term is
\begin{align*}
\text{I} 		&= (p-2) \int_{0}^{m} \lambda^{p-1} \mu(\{ M_t M_x g \ge \lambda/3\}) \, d \lambda .
\end{align*}
To handle $I$, we first notice that 
$$\{ M_t M_x g \ge \lambda/3\} \subset \{M_t M_x (g\mathbf{1}_{\{g > \lambda/6\}}) \ge \lambda/6\}.$$
From this inclusion and the weak type $(\tfrac{3}{2},\tfrac{3}{2})$-bound for the iterated maximal function (which follows from the strong type $(\tfrac{3}{2},\tfrac{3}{2})$), we obtain
$$\mu(\{ M_t M_x g \ge \lambda/3\}) \le \frac{C}{\lambda^{3/2}} \int_{ \{g > \lambda/6\}} g^{3/2} \, d\mu$$
for a dimensional constant $C$. Using this bound in $I$ yields 
\begin{align*}
\text{I} &\le C(p-2)\int_0^m \lambda^{p-5/2}\int_{\{g > \lambda/6\}} g^{3/2} \, d\mu \, d \lambda
\\& = C(p-2) \int_\ree g^{3/2} \int_0^{\min(m, 6g)} \lambda^{p-5/2} \, d\lambda \, d\mu
\\
& = C6^{p-3/2}\frac{p-2}{p-3/2} \int_\ree g_{m/6}^{p-3/2} g^{3/2} \, d \mu
\\
& \leq C6^{p-3/2}\frac{p-2}{p-3/2} \int_\ree g_{m}^{p-2}\, g^{2} \, d \mu.
\end{align*}
Choosing $p - 2>0$ small enough, depending on $A$ and dimension, we see from \eqref{eq1:Gehring} that
\[\int_{\ree} g_{m}^{p-2} g^{2} \, d \mu \leq \frac{1}{2} \int_{\ree} g_{m}^{p-2} g^{2} \, d \mu +  C \Big(\|f\|_{\L^p(\mathbb{R}^{n+1})}^p + \|h^{s}\|_{\L^{q_\alpha}(\mathbb{R}; \L^{q_\beta}(\mathbb{R}^{n}))} ^{p/s} \Big) \]
for some constant $C$ depending on $\alpha, \beta, q_\alpha, q_\beta$ and dimension. This finishes the proof after simplifying the first term and taking the limit $m \to \infty$.
\end{proof}

\begin{rem}
The same estimate also holds with the mixed norm in different order as we are free to interchange the fractional maximal functions.
If we want $q_\alpha = q_\beta$, then \eqref{eq:a,b,qa,qb} reveals that $\alpha, \beta$ are uniquely determined by $s,p$ and $n$.  Hence, for each $p$ there is at most one such pair.
\end{rem}

In the application to our parabolic PDE, we consider special values for the auxiliary parameters in Lemma~\ref{lem:global}.

\begin{cor}
\label{cor:rh} 
Suppose the setup of Lemma~\ref{lem:global} with $s = 2_*$. Then for $p>2$ with $p-2$ small enough depending on $A$ and dimension, 
\begin{align*}
  \|g\|_{\L^{p}(\mathbb{R}^{n+1})} \lesssim  \|f\|_{\L^{p}(\mathbb{R}^{n+1})} + \|h\|_{\L^{p_{*}}(\mathbb{R}^{n+1})}. 
\end{align*}
The implicit constant depends on $A$, $p$ and $n$.
\end{cor}

\begin{proof} We have $s = \frac{2n+4}{n+4}$. We want $q_{\alpha}=q_{\beta}$ in Lemma~\ref{lem:global}, and so we can solve in \eqref{eq:a,b,qa,qb} for
\begin{align*}
 q_\alpha = q_\beta = \frac{p(n+2)}{s(p+n+2)}
\end{align*}
corresponding to
\begin{align*}
 \alpha = \frac{1}{q_\alpha} - \frac{s}{p}, \quad \beta = \frac{n}{q_\beta} - \frac{ns}{p}.
\end{align*}
Indeed, we have $\alpha, \beta > 0$ due to $q_\alpha = q_\beta < p/s$ and $q_\alpha = q_\beta > 1$ follows from 
\begin{align*}
 s = \frac{2(n+2)}{{n+4}} < \frac{p(n+2)}{{n+p+2}} = q_\alpha s = q_\beta s,
\end{align*}
since we have $p>2$.
\end{proof}
\section{Higher integrability of the parabolic differential: Real analysis proof}
\label{sec:proof1}

Let $\Omega = I_0 \times Q_0$ be the ambient parabolic cylinder and $u$ a weak solution to \eqref{eq1} in $\Omega$. Given $\chi \in \C_0^\infty(\Omega)$ we know from Section~\ref{sec:i} that under suitable assumptions on $f$ and $F$ the (localized) parabolic differential
\begin{align*}
 |\nabla(u \chi)| + |\dhalf(u \chi)|
\end{align*}
belongs to $\L^2(\ree)$. In this section, we give a first proof of the following higher integrability result, which lies at the heart of our considerations. Since the Hilbert transform is isometric on $\L^2(\R)$ it does not matter whether or not we include $\HT \dhalf(u \chi)$ here but for $\L^p$-results it seems appropriate to treat both half-order derivatives.

\begin{thm}
\label{thm:reg} 
Suppose $p>2$ is sufficiently close to $2$, depending only on ellipticity and dimensions. If $u$ is a weak solution in $\Omega$ to $\partial_t u - \div A(t,x) \nabla u = f + \div F$, where $f \in \Lloc^{p_*}(\Omega; \IC^m)$ and $F \in \Lloc^{p}(\Omega; \IC^{mn})$, then 
\begin{align*}
 |\nabla (\chi u)| + |\dhalf (\chi u)| + |\HT \dhalf (\chi u)| \in \L^p(\ree)
\end{align*}
for any $\chi \in \C_0^\infty(\Omega)$.
\end{thm}

In this section we want to give a proof using the Gehring lemma with tail. The required non-local reverse H\"older estimate is provided by the following key lemma. Naturally its proof is somewhat technical and will be postponed until the end of this section. It can be skipped on a first reading.

\begin{lem}\label{lem:rhrws} Let $\widetilde f\in \L^{2_*}(\ree)$ and $\widetilde F\in \L^2(\ree)$. Let $v\in V$ be a weak solution to $\partial_t v -\div A(t,x)\nabla v  =  \widetilde f+ \div \widetilde  F$ in $\ree$. Let $\gamma>1$ and let $I \times Q$ be a parabolic cylinder with $\ell(I) \sim r(Q)^2$. Then
\begin{align*}
 g :=  |\nabla v | + |\dhalf v| + |\HT\dhalf v|.
\end{align*}
satisfies
\begin{align}
\label{eq:rhrws}
\begin{split}
\bigg(\bariint_{I \times Q} g^2  \bigg)^{1/2} & 
\lesssim \sum_{k \in \IZ} \frac{1}{1+|k|^{3/2}} \left(\bariint_{I_k \times \gamma Q} g 
+ \bigg(\bariint_{I_k \times \gamma Q} |\widetilde F|^2\bigg)^{1/2} 
+ r(Q) \bigg(\bariint_{I_k \times \gamma Q} |\widetilde f|^{2_*}\bigg)^{1/2_*} \right).
\end{split}
\end{align}
Here,   $I_k := k \ell(I) + I$ are the disjoint translates of $I$ covering the real line up to a countable set. The implicit constant depends only on ellipticity, dimensions, $\gamma$ and the constants controlling the ratio $r(Q)^2/\ell(I)$. 
\end{lem}

For the moment, let us admit the lemma and record its consequences.

\begin{cor}\label{cor:main} Let $\widetilde f \in \L^{2_*}(\ree)$, $\widetilde F\in \L^2(\ree)$ and let $v\in V$ be a weak solution to $\partial_t v -\div A(t,x)\nabla v  =  \widetilde f+ \div \widetilde  F$  in $\ree$. Let $g$ be as in Lemma~\ref{lem:rhrws}. If $p>2$ is sufficiently close to $2$, then
$$
\| g\|_{\L^p(\ree)}  \lesssim   \|  \widetilde F\|_{\L^p(\ree)} +  \|  \widetilde f\|_{\L^{p_{*}}(\ree)}.
$$
The implicit constant as well as $p$ depends only on ellipticity and dimensions.
\end{cor}

\begin{proof}
Rearranging unions of translates of an interval $I$ into unions of its dilates, and vice versa, reveals that for any positive function $h$ on the real line we have 
$$
\sum_{k \in \IZ} \frac{1}{1+|k|^{3/2}} \barint_{  I_k} h \sim \sum_{j = 0}^\infty 2^{-j} \barint_{  4^{j}I}h,
$$
with absolute implicit constants. Lemma~\ref{lem:rhrws} together with this observation and H\"older's inequality yields
\begin{align*}
\bigg(\bariint_{I \times Q} g^2  \bigg)^{1/2} 
\lesssim \sum_{j=0}^\infty 2^{-j} \bariint_{4^j I \times Q} g 
 + \bigg(\sum_{j =0}^\infty 2^{-j} \bariint_{4^j I \times Q} |\widetilde F|^2\bigg)^{1/2} 
 + r \bigg(\sum_{j=0}^\infty 2^{-j} \bariint_{4^j I \times Q} |\widetilde f|^{2_*}\bigg)^{1/q'}.
 \nonumber
\end{align*}
Thus, we have the setup of Lemma~\ref{lem:global} with $s = 2_*$ and we conclude by Corollary~\ref{cor:rh}. 
\end{proof}

Theorem~\ref{thm:reg} is obtained through a by now well-known localization procedure.

\begin{proof}[Proof of Theorem~\ref{thm:reg}]
The function $v:= u\chi \in V$ is a weak solution to $\partial_t v -\div A(t,x)\nabla v = \widetilde f+\div \widetilde F$ on $\ree$ with the relations \eqref{eq:tilde}. We can \emph{a priori} assume $2<p \leq 2^*$ and hence have $v \in \L^p(\ree)$ from Theorem~\ref{prop:1}. Then we have $\Lloc^2(\ree) \subseteq \Lloc^{p_*}(\ree)$ since $p_* \leq 2$. This being said, $\widetilde f \in \L^{p_{*}}(\ree)$, $\widetilde F\in \L^p(\ree)$ follows from \eqref{eq:tilde} and the hypotheses on $f$, $F$. Hence, Corollary~\ref{cor:main} applies and the claim follows.
\end{proof}

We turn to the proof of Lemma~\ref{lem:rhrws}. We follow the argument presented in Section~8 of \cite{AEN} for $f=0$ and $F=0$. We omit duplicated arguments but give all other details so that the reader does not have to work through any other section of \cite{AEN}. In this reference, {the order of variables was $(x,t)$ and} an additional spatial dimension was carried through the argument, both for the purpose of treating boundary value problems. The latter plays no role here and can be ignored. Next, $u$ in \cite{AEN} has become $v$ here and the extra property $\dhalf v\in \L^2(\ree)$ provided by Proposition~\ref{prop:1} means that $v$ is a reinforced weak solution in the terminology there.

\begin{proof}[Proof of Lemma~\ref{lem:rhrws}] We remark that $g\in \L^2(\ree)$ due to Proposition~\ref{prop:1} and the fact that $\HT$ is isometric on $\L^2(\ree)$. It suffices to prove the claim for $\gamma=8$ since \emph{a posteriori} a covering argument, which we leave to the reader, gives us the inequality with any $\gamma>1$.

For simplicity, we are also going to assume $r(Q) \sim 1$ and that $I \times Q$ is centered at $(0,0)$ as scaling and translating give us back the general estimate. Having normalized to scale $1$, averages are integrals (up to numerical constants). For the time being it will be enough to work with $\gamma=4$, so that the parabolic enlargement is $16I \times 4Q$. We fix a smooth cut-off $\eta: \ree \to [0,1]$ with support in $4I \times 2Q$ that is $1$ on an enlargement $\frac{9}{4}I \times \frac{3}{2} Q$. For a reason which will become clear later on, we choose $\eta$ to have the product form
\begin{align*}
 \eta(s,y) = \eta_I(s) \eta_Q(y) ,
\end{align*}
where $\eta_I$ is symmetric about $0$ (the midpoint of $I$). We also give a name to the translation sums
\begin{align*}
 \sum(h) := \sum_{k \in \IZ} \frac{1}{1+|k|^{3/2}} \iint_{I_{k} \times 4Q} |h|.
\end{align*}
  
\medskip

\noindent \textbf{Step 1: The spatial average.} The estimate \eqref{eq:Cac1} with $v-c$ and  $c:= \bariint_{I \times 2Q} v$ yields
\begin{align*}
    \left(\iint_{4I\times 2Q} |\nabla v|^2  \right)^{1/2} \lesssim \iint_{16I\times 4Q} |v-c| + \left(\iint_{16I\times 4Q} |\widetilde F|^2\right)^{1/2} +  \left(\iint_{16I\times 4Q} |\widetilde f|^{2_*}\right)^{1/2_*}.
\end{align*}
Now, we write
\begin{align*}
 v-c = \bigg(v - \barint_{2Q} v \; \d x\bigg) + \bigg(\barint_{2Q} v \; \d x - \bariint_{I \times 2Q} v \; \d x \, \d t\bigg)
\end{align*}
and apply the $\L^1$-Poincar\'e's inequality to the first term and treat the second term (which does not depend on $x$) via the fractional Poincar\'e inequality in the following lemma with $p=q=1$. Details are written out in the proof of Lemma~8.4 in \cite{AEN}.

\begin{lem}[{\cite[Lem.~8.3]{AEN}}]
\label{lem:fracPoin}
Let $p,q \in [1,\infty)$ satisfy $p/2 < q \leq p$. Then for each interval $J \subset \R$ and every $h \in \H^{1/2}(\R)$,
\begin{align*}
 \bigg(\barint_J \Big|h - \barint_J h \Big|^p \d s \bigg)^{1/p} \lesssim \sqrt{\ell(J)} \bigg(\sum_{k \in \IZ} \frac{1}{1+|k|^{3/2}} \barint_{J_k} |\dhalf h|^q \d s \bigg)^{1/q}.
\end{align*}
The analogous inequality with $\HT \dhalf h$ instead of $\dhalf h$ on the right-hand side also holds.
\end{lem}
The resulting estimate is 
\begin{equation}
\label{eq:sigma}
\iint_{16I\times 4Q} |v-c| \lesssim  \sum(g).
\end{equation}
Thus, we obtain a bound of the required type
\begin{align}
\label{eq:RHS}
\begin{split}
    \left(\iint_{4I\times 2Q} |\nabla v|^2  \right)^{1/2} 
    &\lesssim \sum(g) + \left(\iint_{16I\times 4Q} |\widetilde F|^2\right)^{1/2} +  \left(\iint_{16I\times 4Q} |\widetilde f|^{2_*}\right)^{1/2_*} \\
    &\lesssim \sum(g) + \sum_{|k| \leq 16} \frac{1}{1+|k|^{3/2}} \left(\bigg(\iint_{I_k \times 4Q} |\widetilde F|^2\bigg)^{1/2} + \bigg(\iint_{I_k \times 4Q} |\widetilde f|^{2_*}\bigg)^{1/2_*}\right).
\end{split}
\end{align}
It remains to obtain akin bounds for the $\L^2$ averages of $\HT\dhalf v$ and $\dhalf v$. As the fractional derivatives annihilate constants, we may replace $v$ by $v-c$ and write $v-c=\eta(v-c)+ (1-\eta)(v-c)$. 

\medskip

\noindent \textbf{Step 2: Local terms.} For the local term $w:=\eta(v-c)$ we have   
\begin{align*}
\iint_{I \times Q} |\HT\dhalf w|^2+  |\dhalf w|^2 \leq \iint_{\ree} |\HT\dhalf w|^2+  |\dhalf w|^2= 2 \iint_{\ree} 
|\dhalf w|^2,  
\end{align*}
using that $\HT$ is isometric on $\L^2(\ree)$. Since $w$ solves an equation $\pd_{t}w- \div A \nabla w = f'+\div  F'$, Proposition~\ref{prop:1} implies
\begin{align*}
\|\dhalf w\|_{2}\le \|w\|_{V}+ \| f'\|_{2_*}+ \| F'\|_{2},
\end{align*}
where
\begin{align*}
\label{}
    f' &  := \eta \widetilde  f + \pd_{t}\eta (v-c) - A\nabla v \cdot \nabla \eta - \widetilde F\cdot \nabla \eta,   \\
    F' &  :=  -A ((v-c) \nabla \eta)  + \widetilde F\eta,
\end{align*}
as we can read off from \eqref{eq:tilde}. Using the formul\ae{} for $f', F'$ and H\"older's inequality to bound $\L^{2_*}$ averages by $\L^2$ averages whenever convenient, we arrive at
\begin{align*}
\|\dhalf w\|_{2}  \lesssim \left( \iint_{4I\times 2Q}|v-c|^{2}\right)^{1/2} + \left( \iint_{4I\times 2Q}|\nabla v|^2\right)^{1/2} + \left( \iint_{4I\times 2Q}|\widetilde F|^2\right)^{1/2} + \left( \iint_{4I\times 2Q}|\widetilde f|^{2_*}\right)^{1/2_*}.
\end{align*}
For the first term on the right we use \eqref{eq:rhu} and then \eqref{eq:sigma}. For the second one we use \eqref{eq:RHS}.
This leads to
\begin{align*}
    \|\dhalf w\|_{2}  \lesssim \sum(g) + \left(\iint_{16I \times 4Q} |\widetilde F|^2\right)^{1/2} +  \left(\iint_{16I\times 4Q} |\widetilde f|^{2_*}\right)^{1/2_*}
\end{align*}
and decomposing $16I$ into translates of $I$ as before gives an estimate of the required type.

\medskip

\noindent \textbf{Step 3: First error term.} We come to the delicate steps in \cite{AEN}. The non-locality of the operators $\dhalf$ and $\HT \dhalf$ cannot be circumvented anymore. As $\eta_{Q}=1$ on $Q$, we have 
\begin{align*}
 \dhalf \big((1-\eta)(v-c)\big)=\dhalf \big((1-\eta_{I})(v-c)\big)
\end{align*}
on $I \times Q$. The same observation applies to $\HT \dhalf$ in lieu of $\dhalf$. We split as in Step~1,
\begin{align*}
 v-c=v-\barint_{2Q} v + \barint_{2Q}v- \bariint_{I \times 2Q} v.
\end{align*}
For the terms involving $w_{1}:=(1-\eta_{I})(v-\barint_{2Q} v)$ we can use a kernel representation for $\dhalf$ and then the fractional Poincar\'e inequality of Lemma~\ref{lem:fracPoin}. This is (the proof of) Lemma~8.6 in \cite{AEN}. As a result,
\begin{align*}
\left( \iint_{I \times Q} |\HT\dhalf w_{1}|^2+  |\dhalf w_{1}|^2 \right)^{1/2} &\lesssim \sum_{j\in \IZ}  \frac{1}{1+|j|^{3/2}} \bigg(\iint_{4I_j\times 2Q} |\nabla v|^2\bigg)^{1/2}. 
\end{align*} 
Inserting \eqref{eq:RHS} for each $4I_j\times 2Q$ instead of $4I \times 2Q$ and using the convolution inequality 
\begin{align*}
 \sum_{j \in \IZ} \frac{1}{1+|j|^{3/2}}\frac{1}{1+|k-j|^{3/2}} \lesssim \frac{1}{1+|k|^{3/2}}, \qquad k \in \IZ,
\end{align*}
we obtain the desired bound by
\begin{align*}
  \sum(g)+  \sum_{k\in \IZ}  \frac{1}{1+|k|^{3/2}} \bigg(\iint_{I_k \times 4Q} |\widetilde F|^2\bigg)^{1/2} +  \sum_{k\in \IZ}  \frac{1}{1+|k|^{3/2}} \bigg(\iint_{I_k \times 4Q} |\widetilde f|^{2_*}\bigg)^{1/2_*}.
\end{align*}

\medskip

\noindent \textbf{Step 4: Second error term.} The remaining average of $|\HT\dhalf w_{2}|^2+  |\dhalf w_{2}|^2$, where 
\begin{align*}
 w_{2}:=(1-\eta_{I})\bigg(\barint_{2Q}v- \bariint_{I \times 2Q} v \bigg) =: (1-\eta_{I})\bigg(h - \barint_I h\bigg), \quad h = \barint_{2Q}v,
\end{align*}
is treated independently of knowing that $v$ is a solution. Indeed, since $v \in \H^{1/2}(\R; \L^2(\R^n))$ we have $h, (1-\eta_I)h \in \H^{1/2}(\R)$ and we had chosen $\eta_{I}$ in such a way that the following lemma applies. We note that in its proof the symmetry of $\eta_{I}$ is used to control the Hilbert transform, which has an odd kernel.

\begin{lem}[{\cite[Lem.~8.7]{AEN}}]
\label{lem:cotlar}
Let $I$ be a bounded interval and $\eta_I: \R \to [0,1]$ be a smooth cut-off function with support in $4I$ that is identically $1$ on $\frac{9}{4} I$. Suppose furthermore that $\eta_I$ is symmetric about the midpoint of $I$. If $h, (1-\eta_I)h  \in \H^{1/2}(\R)$, then almost everywhere on $I$,
\begin{align}
\label{eq:cotlar}
 \bigg|\HT \dhalf \bigg((1-\eta)\Big(h - \barint_I h \Big)\bigg)\bigg|  \lesssim \sum_{k \in \IZ} \frac{1}{1+|k|^{3/2}} \barint_{I_k} (|\dhalf h| + |\HT \dhalf h| ).
\end{align}
\end{lem}

We take the $\L^2(I \times Q)$ average on both sides of \eqref{eq:cotlar} and use that $\HT \dhalf$ commutes with averages in the spatial variable to give
\begin{align*}
 \left( \bariint_{I \times Q} |\HT\dhalf w_{2}|^2 \right)^{1/2} \lesssim   \sum_{k\in \IZ}  \frac{1}{1+|k|^{3/2}} \bariint_{I_k \times 2Q} (|\dhalf v|+ |\HT\dhalf v|). 
\end{align*}
This completes the required bound for the $\L^2(I \times Q)$ average of $|\HT\dhalf v|^2$ with $\gamma=4$ on the right-hand side.

It only remains to consider the bound for the $\L^2(I \times Q)$ average of $\dhalf w_{2} = \dhalf ((1-\eta_I)(h - \barint_I h))$. This can be done by another lemma on real functions from \cite{AEN}, with a weaker conclusion since $\dhalf$ has an even kernel.

\begin{lem}[{\cite[Rem.~8.10]{AEN}}]
Under the assumptions of Lemma~\ref{lem:cotlar} it holds
\begin{align*}
 \left(\barint_I \bigg|\dhalf \bigg((1-\eta)\Big(h - \barint_I h \Big)\bigg)\bigg|^2 \right)^{1/2}  \lesssim \bigg(\barint_{4I} |\HT \dhalf h|^2 \bigg)^{1/2} + \sum_{k \in \IZ} \frac{1}{1+|k|^{3/2}} \barint_{I_k} (|\dhalf h| + |\HT \dhalf h|).
\end{align*}
\end{lem}

Indeed, for $h = \barint_{2Q}v$ as before, we obtain
\begin{align*}
\left( \bariint_{I \times Q} |\dhalf w_{2}|^2 \right)^{1/2} \lesssim   \bigg(\bariint_{4I \times 2Q} |\HT \dhalf v|^2 \bigg)^{1/2} + \sum_{k\in \IZ}  \frac{1}{1+|k|^{3/2}} \bariint_{I_k \times 2Q} (|\dhalf v|+ |\HT\dhalf v|). 
\end{align*}

The upshot is that we have already completed the reverse H\"older estimate for $|\HT \dhalf v|^2$ on any parabolic cylinder, and in particular on $4I \times 2Q$, with spatial enlargement by a factor $4$ on the right-hand side.  Hence, if we finally use $\gamma = 8$ on the right-hand side of \eqref{eq:rhrws}, then the above estimate completes the reverse H\"older bound for $\dhalf v$.
\end{proof}

\section{Higher integrability of the parabolic differential: Operator theoretic proof}
\label{sec:proof2}

We provide a second proof for the higher integrability of the parabolic differential using a completely different method. 
For $1<p<\infty$,  we set 
\begin{align*}
 E_{p}:= \L^p(\R; \W^{1,p}(\R^n)) \cap \H^{1/2, p}(\R; \L^p(\R^n))
\end{align*}
with norm $\|u\|_{E_{p}} := (\|u\|_{p}^p + \|\nabla u \|_{p}^p+ \|\dhalf u\|_{p}^p)^{1/p}$, so that in particular $E=E_{2}$ is as in Section~\ref{sec:i}.  These are Banach spaces with $ \C^\infty_{0}(\ree)$ as a common dense subspace as is seen again by approximation via smooth convolution and truncation.

We shall use some results on complex interpolation of Banach spaces. The reader will find necessary background information in the textbook~\cite{BL}.  For the understanding of this paper it will be enough to know the complex interpolation identity
\begin{align*}
 [E_{p_0}, E_{p_1}]_\theta = E_p, \qquad p_0,p_1 \in (1,\infty), \quad \frac{1-\theta}{p_0} + \frac{\theta}{p_1} = \frac{1}{p},
\end{align*}
in the sense of Banach spaces with equivalent norms. We shall say that $(E_p)_{1<p<\infty}$ is a complex interpolation scale and the same holds true for the dual scale $(E_{p'}^*)_{1<p<\infty}$. These assertions were proved in \cite[Lem.~6.1]{AEN}.

We then have the following extension of Lemma~\ref{lem:2}.

\begin{lem}\label{lem:5}  The operator  $\cL= \partial_t -\div A(t,x)\nabla  + \kappa + 1:E_{2}\to E_{2}^*$ extends by density to a bounded operator from $E_{p}$ to  $(E_{p'})^*$ for $1<p<\infty$. This extension is invertible for $|p-2|$ small enough and its inverse agrees with the one calculated when $p=2$ on $(E_{2})^*\cap (E_{p'})^*$. The norm of the inverse and the smallness of $|p-2|$ depend only on ellipticity and dimensions.  
\end{lem}

\begin{proof} 
By definition, $\cL: E_{2}\to E_{2}^*$ acts via
\begin{align*}
 \langle \cL u, v \rangle = \iint_\ree A \nabla u \cdot \cl{\nabla v} + \HT \dhalf u \cdot \cl{\dhalf v} + (\kappa + 1)u \cdot \cl{v} \; \d x \, \d t.
\end{align*}
Thus, $E_{p} \to (E_{p'})^*$ boundedness of $\cL$ follows from H\"older's inequality and the norm depends only on ellipticity and dimension. As $\cL$ is invertible when $p=2$ by Lemma~\ref{lem:2}, the invertibility for $|p-2|$ small enough follows from  {\v{S}}ne{\u\ii}berg's result on bounded operators acting on interpolation scales, see \cite{Sneiberg-Original} or \cite[Thm.~A.1]{ABES3} for a qualitative version revealing that the smallness of $|p-2|$ and the bound for the inverse depend only on ellipticity and dimensions. Finally, the compatibility of the inverses is an abstract feature of complex interpolation, see Theorem~8.1 in \cite{KMM}.
\end{proof}

A simple but important consequence is

\begin{lem}
\label{lem:6} 
Let $\widetilde f \in \L^{p_{*}}(\ree)$ and $\widetilde F\in \L^p(\ree)$. Then $\cL^{-1}(\widetilde f+ \div \widetilde F)\in E_{p}$ when $|p-2|$ is sufficiently small (depending only on ellipticity and dimensions) and in this case
\begin{align*}
 \|\cL^{-1}(\widetilde f+ \div \widetilde F)\|_{E_p} \lesssim \|\widetilde f\|_{\L^{p_{*}}(\ree)} + \|\widetilde F\|_{\L^p(\ree)},
\end{align*}
with an implicit constant depending only on ellipticity and dimensions.
\end{lem}

\begin{proof} 
From Lemma~\ref{lem:1}, $E_{p'}$ embeds into $\L^{{p'}^*}(\ree)$ when $1<p'<n+2$. As the dual exponent of ${p'}^*$ is $p_{*}$, we obtain that $\L^{p_{*}}(\ree)$ embeds into $(E_{p'})^*$.  Thus $\widetilde f+ \div \widetilde F\in (E_{p'})^*$ and the conclusion follows from Lemma~\ref{lem:5}.
\end{proof}

With this at hand, we are ready to give the second

\begin{proof}[Proof of Theorem~\ref{thm:reg}] Let $\chi\in \C_{0}^\infty(\Omega)$.  As before,  $v:=u\chi \in V$  is a weak solution to $\partial_t v -\div A(t,x)\nabla v  = \widetilde f+\div \widetilde F $ on $\ree$ with $\widetilde f, \widetilde F$ given by 
\eqref{eq:tilde}. By Proposition~\ref{prop:1} we know that $v\in E_2$ and $\cL v = (\kappa + 1)v + \widetilde f+\div \widetilde F$.

Let now $p>2$ be such that we have Lemma~\ref{lem:6} at our disposal. We may also suppose $p \leq 2^*$, which is equivalent to $p_* \leq 2$. If we assume $f\in \Lloc^{p_{*}}(\Omega)$ and $ F\in \Lloc^{p}(\Omega)$, then $\widetilde f  + (\kappa +1) v \in \L^{p_{*}} (\ree)$ and $\widetilde F\in \L^p(\ree)$, using also Theorem~\ref{thm:reg-i} to control $\|u \nabla \chi\|_{2^*}$ in the formula for $\widetilde F$. Hence, by compatibility of the inverses (Lemma~\ref{lem:5}) and Lemma~\ref{lem:6}, we obtain $v\in E_{p}$ with 
\begin{align}
\label{eq:second}
 \|v\|_{E_{p}} \lesssim  \|v\|_{p_{*}}+  \|\widetilde f\|_{p_*}+ \|\widetilde F\|_{p}.
\end{align}
The left-hand side controls $\|u \chi\|_{p} + \|\nabla(u\chi)\|_{p}+ \|\dhalf (u\chi)\|_{p}$ and we are done.
\end{proof}
\section{Local higher regularity estimates}
\label{sec:proof3}

Eventually, we shall use the previously obtained qualitative information of higher integrability for the parabolic differential of the localized solution to obtain scale-invariant local higher regularity estimates. This is summarized in the following theorem. As usual, $\Omega = I_0 \times Q_0$ denotes the ambient parabolic cylinder.

\begin{thm}
\label{thm:RH grad} 
Suppose $u$ is a local weak solution in $\Omega$ to $\partial_t u - \div A(t,x) \nabla u = f + \div F$, where $f \in \Lloc^{2_*}(\Omega; \IC^m)$ and $F \in \Lloc^{2}(\Omega; \IC^{mn})$. Let $\gamma>1$ and $I \times Q$ a parabolic cylinder with $\ell(I) \sim r(Q)^2$ such that $\cl{\gamma^2 I \times \gamma Q} \subseteq \Omega$. If $p>2$ is sufficiently close to $2$ (depending only on ellipticity and dimensions), then with $\alpha= 1/2-1/p$,
\begin{align}
\label{eq:rht}
\begin{split}
\bigg(\bariint_{ I \times Q} |\nabla u|^p & \bigg)^{1/p}
+  \sup_{t\in I} \left(\barint_{ Q } | u(t,\blank)|^p \right)^{1/p} 
+ \sup_{t,s\in I} \left(\barint_{ Q } \frac{ | u(t,\blank) - u(s,\blank)|^p}{|t-s|^{\alpha p}} \right)^{1/p} \\ 
& \lesssim \frac{1}{r(Q)} \bigg(\bariint_{\gamma^2 I \times \gamma Q} |u|^2\bigg)^{1/2}
 + \left(\bariint_{\gamma^2 I \times \gamma Q} |F|^p\right)^{1/p}
 + r(Q) \left(\bariint_{\gamma^2 I \times \gamma Q} |f|^{p_{*}}\right)^{1/p_{*}}.
\end{split}
\end{align}
The implicit constant depends only on ellipticity, dimensions, $\gamma$ and the constants controlling the ratio $r(Q)^2/\ell(I)$. 
\end{thm}

\begin{proof} We assume again $r(Q) = 1$ as rescaling will give us the right powers of $r(Q)$. We follow the usual strategy and pick $\chi\in \C_{0}^\infty(\Omega)$, $\chi = 1$ on $I \times Q$, with support in $\gamma^2 I \times \gamma Q$. Then $v:=u\chi \in V$ is a weak solution to $\partial_t v -\div A(t,x)\nabla v = \widetilde f+\div \widetilde F$ on $\ree$ with the relations \eqref{eq:tilde}
and $\widetilde f \in \L^{2_*}(\ree)$, $\widetilde F\in \L^2(\ree)$. By Corollary~\ref{cor:main} we have if $p-2>0$ is small enough,
\begin{align}
\label{eq1:Proof1}
\|\nabla v \|_{p} + \|\dhalf v\|_{p}  \lesssim   \| \widetilde F\|_{p}  
+   \| \widetilde f\|_{p_*}.
\end{align}
Alternatively, we could have used \eqref{eq:second} here at the expense of a term $\|v\|_{p_*} \lesssim \|v\|_p$ on the right, which turns out to be harmless. Indeed, we have from \eqref{eq:rhu} if $p\le 2^*$, as we may assume,
\begin{align}
\label{eq2:Proof1}
\|v\|_{p } \lesssim \|v \|_{2}+ \|\widetilde F\|_{2}+ \|\widetilde f\|_{2_*} \lesssim \|u \chi \|_{2} + \|\widetilde F\|_{p}+ \|\widetilde f\|_{p_*}.
\end{align}
We have used $2_* \leq p_*$ and $2<p$ in the second step.

We have shown that $v, \dhalf v$ are controlled in $ \L^p(\ree)$ and in $\L^2(\ree)$. Since $p>2$, a H\"older norm estimate on $v$ will follow from classical embeddings. We approximate $v$ through convolution with smooth kernels in the $t$-variable, say $v_\eps = v \ast_t \varphi_\eps$, where $\eps >0$.  For almost every $x \in \R^n$ we can apply the fractional Poincar\'e inequality from Lemma~\ref{lem:fracPoin} to $v_\eps(\blank,x)$. Hence, for any interval $J \subset \R$, and $\alpha=1/2-1/p$, we have
$$
\left( \barint_{J} \Big|v_\eps(\blank ,x) - \barint_{J}v_\eps(\blank,x)\Big|^p \right)^{1/p} \lesssim \ell(J)^\alpha \|\dhalf v_\eps(\blank,x)\|_{\L^p(\R)},
$$
and the Campanato characterization of H\"older regularity yields
\begin{align*}
 \sup_{t \in J} |v_\eps(t,x)| + \sup_{t,s \in J} \frac{|v_\eps(t,x) - v_\eps(s,x)|}{|t-s|^\alpha} \lesssim \|v_\eps(\blank,x)\|_{\L^p(\R)} +  \|\dhalf v_\eps(\blank,x)\|_{\L^p(\R)},
\end{align*}
where the implicit constant depends also on $\ell(J)$, see Theorem~2.9 in \cite{Giusti}. We take $J=I$, average the $p$-th power of this estimate over $x \in Q$ and then we can pass to the limit as $\eps \to 0$. This reveals that the left hand side of \eqref{eq:rht} is bounded by $\|v\|_{p}+ \|\nabla v\|_{p} + \|\dhalf v\|_p$. (The reader should recall the normalization $r = 1\sim \ell$ and the construction of $\chi$.) In view of \eqref{eq1:Proof1} and \eqref{eq2:Proof1}, we see that it remains to control $ \| \widetilde F\|_{p}  +   \| \widetilde f\|_{p_*}$ from above by the right-hand side of \eqref{eq:rht}.

We begin with $\widetilde F$. Let $1<\delta<\gamma$ be such that the support of $\chi$ is contained in $\delta^2 I\times \delta Q$. By \eqref{eq:tilde} we have,
\begin{align*}
\| \widetilde F\|_{p}    \lesssim \left(\iint_{\gamma^2 I \times \gamma Q} |F|^p\right)^{1/p} +
 \left(\iint_{\delta^2 I\times \delta Q} |u|^p\right)^{1/p}  
\end{align*}
and since we already assumed $2\leq p\le 2^*$ (which implies $2_* \leq p_* \leq 2$), we can use \eqref{eq:rhu} to conclude
\begin{align*}
 \| \widetilde F\|_{p} \lesssim\left(\iint_{\gamma^2 I \times \gamma Q} |F|^p\right)^{1/p}  + \bigg(\iint_{\gamma^2 I \times \gamma Q} |u|^2\bigg)^{1/2}+ \left(\iint_{\gamma^2 I \times \gamma Q} |f|^{p_*}\right)^{1/{p_*}}.
\end{align*}
Similarly, we use the definition of $\widetilde f$ in \eqref{eq:tilde} to infer
 \begin{align*}
\label{}
\| \widetilde f\|_{p_*}       \lesssim \left(\iint_{\delta^2 I\times \delta Q} |f|^{p_*}+ |u|^{p_*}+ |\nabla u|^{p_*} + |F|^{p_*}\right)^{1/{p_*}}  
\end{align*}
and we are done as $p_{*}\le 2$ and since the term integral over $|\nabla u|^{p_*}$ can be treated using Proposition~\ref{ref:cac}. 
\end{proof}

Finally, we obtain a true reverse H\"older estimate for $\nabla u$, that is to say, an analogous estimate without $u$ on the right-hand side.

\begin{thm}
\label{thm:GS} 
Consider the setup of Theorem~\ref{thm:RH grad} and let $\gamma>1$ and $I \times Q$ be a parabolic cylinder with $\ell(I) \sim r(Q)^2$ such that $\cl{\gamma^2 I \times \gamma Q} \subseteq \Omega$. If $p>2$ is sufficiently close to $2$ (depending only on ellipticity and dimensions), then  
\begin{equation}
 \label{eq:grad}
 \left( \bariint_{I \times Q} |\nabla u|^p\right)^{1/p} \lesssim   \bariint_{\gamma^2 I \times \gamma Q} |\nabla u| + \left(\bariint_{\gamma^2 I \times \gamma Q} |F|^p\right)^{1/p} +  r(Q)  \left(\bariint_{\gamma^2 I \times \gamma Q} |f|^{p_{*}}\right)^{1/p_{*}}.
\end{equation}
The implicit constant depends on  ellipticity, dimensions, $\gamma$ and the constants controlling the ratio $r(Q)^2/\ell(I)$. 
\end{thm}

\begin{proof} 
As in the proof of Proposition~\ref{ref:basic} the self-improving properties~\cite{BCF, IN} of reverse H\"older estimates yield the conclusion once we have managed to prove \eqref{eq:grad} with an $\L^{2}$ average of $|\nabla u|$ on the right-hand side. By scaling it is also enough to assume $r(Q) = 1$.
    
We use  the ``weighted means trick'' introduced by Struwe in \cite{Stru}.
We choose $\chi$ real-valued, equal to $1$ on $I \times Q$ and supported in $\gamma^2 I \times \gamma Q$ of the form $\chi(t,x)=\eta(t)\varphi(x)$. Then define the weighted mean 
\begin{align*}
 \widetilde u(t) := a \int u(t,x) \varphi(x)\; \d x, \quad a:=\bigg(\int \varphi(x) \; \d x \bigg)^{-1}.
\end{align*}
We set $w(t,x):=(u(t,x)-\widetilde u(t))\eta(t)$.  We remark that $\nabla u = \nabla w$. It is thus enough to estimate $ \big( \iint_{I \times Q} |\nabla w|^p\big)^{1/p}$. We proceed as follows. 
    
It follows from the equation for $u$ that $\widetilde u $ is absolutely continuous on $I$ with 
\begin{align}
\label{eq:pdtildeu}
 \partial_t \widetilde u(t)  =   a \int -(A \nabla u +F)\cdot \nabla \varphi + f\varphi \; \d x
\end{align}
almost everywhere. Since $\eta$ and $\widetilde u$ depend only on $t$, we have $\partial_t w -\div A(t,x)\nabla w  = f' +\div (\eta F)$ in $\Omega$, where, omitting the variables except for the integration, 
\begin{align*}
 f' &:= \eta \Big( f-a \int  f\varphi \; \d x\Big) +   a\eta \int   (A \nabla u - F)\cdot  \nabla \varphi\; \d x + (u-\widetilde u)\pd_{t} \eta.
\end{align*}
Theorem~\ref{thm:RH grad} applied to $w$ yields
\begin{align*}
 \bigg( \iint_{I \times Q} |\nabla w|^p\bigg)^{1/p} \lesssim \bigg(\bariint_{\gamma^2 I \times \gamma Q} |w|^2\bigg)^{1/2}
 + \left(\bariint_{\gamma^2 I \times \gamma Q} |F|^p\right)^{1/p}
 + \left(\bariint_{\gamma^2 I \times \gamma Q} |f'|^{p_{*}}\right)^{1/p_{*}}.
\end{align*}
Now, we insert the definition of $f'$ and estimate all averages crudely using the triangle inequality and the support properties of $\eta$ and $\varphi$. In this manner, we arrive at
\begin{align*}
 \bigg( \iint_{I \times Q} |\nabla w|^p\bigg)^{1/p} 
 \lesssim 
 \bigg(\bariint_{\gamma^2 I \times \gamma Q} |w|^2\bigg)^{1/2}
 + \left(\bariint_{\gamma^2 I \times \gamma Q} |F|^p\right)^{1/p}
 + \left(\bariint_{\gamma^2 I \times \gamma Q} |f|^{p_{*}} + |w|^{p_*} + |\nabla w|^{p_*} \right)^{1/p_{*}}.
\end{align*}
Of course, we may assume $p \leq 2^*$, which is equivalent to $p_* \leq 2$ and hence allows us to bound $\L^{p_*}$ averages of $|\nabla w| = |\nabla u|$ and $|w|$ by the corresponding $\L^2$ averages. As mentioned previously, it suffices to prove \eqref{eq:grad} with an $\L^2$ average of $|\nabla u|$ on the right hand side. So, we are left with controlling the $\L^2$ average of $w$. Since taking weighted averages $u \mapsto \widetilde u$ defines a projection from $\W^{1,2}(\gamma Q)$ onto $\IC$, a variant of Poincar\'e's inequality on $\L^2(\gamma Q)$ discussed for example in \cite[Lem.~8.3.1]{Adams-Hedberg} yields
\begin{align*}
 \int_{\gamma^2 I \times \gamma Q} |w|^2 \leq \int_{\gamma^2 I} \int_{\gamma Q} |u-\widetilde u|^2 \; \d x \, \d t \lesssim \int_{\gamma^2 I} \int_{\gamma Q} |\nabla u|^2 \; \d x \, \d t.
\end{align*}
The proof is complete.
\end{proof}
 
\begin{rem}
 Once  Theorem~\ref{thm:reg} is established, it is also possible to prove directly  \eqref{eq:grad} under our assumptions by adapting the original argument in \cite{GS} and invoking the usual Gehring lemma.
\end{rem}
\appendix
{
\section{Extension of weakly elliptic coefficients}

We provide here a simple lemma justifying the use of the global G\aa rding inequality in the context of local weak solutions.

\begin{lem}
\label{lem::extendGarding}
Let $Q_0 \subset \R^n$ be an open set. Let $A \in \L^\infty(Q_0; \Lop(\IC^{nm}))$ and suppose that there exist $\lambda > 0$ and $\kappa \geq 0$ such that 
\begin{align*}
  \Re \int_{Q_0} A \nabla u \cdot  \con{\nabla u} \geq \lambda \int_{Q_0} |\nabla u|^2 - \kappa  \int_{Q_0} | u|^2, \qquad u \in \W_0^{1,2}(Q_0; \IC^m).
\end{align*}
Let $Q$ be a compact subset of $Q_0$. If $\sigma > 0$ is sufficiently large, depending only on $\lambda$, $\|A\|_\infty$, $n$, $m$ and the distance $\mathrm{dist}(Q, \R^n \setminus Q_0)$, then $\widetilde A := \mathbf{1}_{Q_0} A + \sigma \mathbf{1}_{\R^n \setminus Q}$ satisfies for some constant $K$ depending on the same parameters,
\begin{align*}
  \Re \int_{\R^n} \widetilde A \nabla u \cdot  \con{\nabla u} \geq \frac{\lambda}{4} \int_{\R^n} |\nabla u|^2 - K  \int_{\R^n} | u|^2, \qquad u \in \W^{1,2}(\R^n; \IC^m).
\end{align*}
\end{lem}

\begin{proof}
Let $\varphi: \R^n \to [0,1]$ a smooth cut-off that is $1$ on $Q$, has support in $Q_0$ and satisfies $\|\nabla \varphi\|_\infty \leq \frac{c}{\mathrm{dist}(Q, \R^n \setminus Q_0)}$ for some dimensional constant $c$. Let $u \in \W^{1,2}(\R^n; \IC^m)$ and split $u = u_1 + u_2$, where $u_1:= \varphi u \in \W^{1,2}_0(Q_0; \IC^m)$ and $u_2 := (1-\chi)u$. Accordingly, we split
\begin{align*}
 \int \widetilde A \nabla u \cdot  \con{\nabla u}
= \int \widetilde A \nabla u_1 \cdot  \con{\nabla u_1} + \int \widetilde A \nabla u_2 \cdot  \con{\nabla u_2} + \int (\widetilde A \nabla u_1 \cdot  \con{\nabla u_2} + \widetilde A \nabla u_2 \cdot  \con{\nabla u_1})
=:\text{I} + \text{II} + \text{III}.
\end{align*}
First, by assumption on $A$ and since $\sigma \geq 0$, we have $\Re \text{I} \geq \lambda \|\nabla u_1\|_2^2 - \kappa \|u_1\|_2^2$. Second, since $u_2$ vanishes on $Q$, we get $\Re \text{II} \geq (\sigma - \|A\|_\infty) \|\nabla u_2\|_2^2$ from the definition of $\widetilde A$. At last, again by definition of $\widetilde A$, we have
\begin{align*}
 \Re \text{III} \geq - 2 \|A\|_\infty \|\nabla u_1\|_2 \|\nabla u_2\|_2 + 2 \sigma \Re \int \nabla u_1 \cdot \con{\nabla u_2}.
\end{align*}
Expanding 
\begin{align*}
 \nabla u_1 \cdot \con{\nabla u_2} = (\varphi \nabla u + u \nabla \varphi) \cdot \con{((1-\varphi) \nabla u - u \nabla \varphi)}
\end{align*}
and making the key observation that $\Re (\varphi \nabla u \cdot \con{(1-\varphi) \nabla u}) = \varphi(1-\varphi) |\nabla u|^2$ is non-negative almost everywhere by the choice of $\varphi$, we see that for some constant $C>0$ depending only on $n$, $m$ and $\mathrm{dist}(Q, \R^n \setminus Q_0)$,
\begin{align*}
 \Re \int \nabla u_1 \cdot \con{\nabla u_2} \geq - C \|u\|_2 \|\nabla u\|_2 - C \|u\|_2^2.
\end{align*}
Summing up, we discover
\begin{align*}
 \int \widetilde A \nabla u \cdot  \con{\nabla u}
 &\geq \lambda \|\nabla u_1\|_2^2 +  (\sigma - \|A\|_\infty) \|\nabla u_2\|_2^2 - 2 \|A\|_\infty \|\nabla u_1\|_2 \|\nabla u_2\|_2 \\
 &\qquad- 2 \sigma C \|u\|_2 \|\nabla u\|_2 - (2 \sigma C + \kappa) \|u\|_2^2.
\end{align*}
Note that $\|\nabla u\|_2^2 \leq 2 \|\nabla u_1\|_2^2 + 2 \|\nabla u_2\|_2^2$ as a consequence of $u = u_1 + u_2$. Hence, we can fix $\sigma$ large enough depending on $\lambda$ and $\|A\|_\infty$ and apply Young's inequality to deduce
\begin{align*}
 \int \widetilde A \nabla u \cdot  \con{\nabla u} \geq \frac{\lambda}{2} (\|\nabla u_1\|_2^2 + \|\nabla u_2\|_2^2) - K \|u\|_2^2,
\end{align*}
where $K$ depends on all the other (by now fixed) parameters. The same estimate on the gradients as before yields the claim.
\end{proof}
}
\def\cprime{$'$} \def\cprime{$'$} \def\cprime{$'$}

\setlength{\parindent}{0pt}
\end{document}